\documentclass[11pt]{amsart}

\usepackage{a4,graphicx,mathrsfs,amssymb, tikz, cancel, color}
\usepackage[colorlinks=true]{hyperref}
\usepackage{enumerate}

\newtheorem{theorem}{Theorem}[section]
\newtheorem{lemma}[theorem]{Lemma}
\newtheorem{corollary}[theorem]{Corollary}

\theoremstyle{definition}
\newtheorem{remark}[theorem]{Remark}

\theoremstyle{remark}

\numberwithin{equation}{section}

\renewcommand{\AA}{\mathscr{A}}
\newcommand{\BB}{\mathscr{B}}

\newcommand{\HH}{\mathscr{H}}

\newcommand{\VV}{\mathcal{V}}

\newcommand{\field}[1]{\mathbb{#1}}

\newcommand{\R}{\field{R}}
\newcommand{\N}{\field{N}}

\renewcommand{\Re}{\mathop{\text{\upshape{Re}}}}

\newcommand{\supp}{\mathop{\rm{supp}}}

\renewcommand{\div}{\mathop{\rm{div}}}

\newcommand{\la}{\lambda}
\newcommand{\rh}{\rho}

\newcommand{\Ga}{\Gamma}
\newcommand{\De}{\Delta}

\newcommand{\Om}{\Omega}

\renewcommand{\bar}[1]{\overline{#1}}
\renewcommand{\tilde}[1]{\widetilde{#1}}

\newcommand{\md}[1]{\color{black}#1\color{black}}

\allowdisplaybreaks

\begin{document}
\title[A Plate-Membrane Transmission Problem]{Regularity and asymptotic behaviour for a damped plate-membrane transmission problem}

\author{Bienvenido Barraza Mart\'inez}
\address{B.\ Barraza Mart\'inez, Universidad del Norte, Departamento de Matem\'aticas y Estad\'istica, Barranquilla, Colombia}
\email{bbarraza@uninorte.edu.co}

\author{Robert Denk}
\address{R.\ Denk, Universit\"at Konstanz, Fachbereich f\"ur Mathematik und Statistik, Konstanz, Germany}
\email{robert.denk@uni-konstanz.de}

\author{Jairo Hern\'andez Monz\'on}
\address{J.\ Hern\'andez Monz\'on, Universidad del Norte, Departamento de Matem\'aticas y Estad\'istica, Barranquilla, Colombia}
\email{jahernan@uninorte.edu.co}

\author{Felix Kammerlander}
\address{F.\ Kammerlander, Universit\"at Konstanz, Fachbereich f\"ur Mathematik und Statistik, Konstanz, Germany}
\email{felix.kammerlander@uni-konstanz.de}

\author{Max Nendel}
\address{M.\ Nendel, Universit\"at Bielefeld, Institut f\"ur Mathematische Wirtschaftsforschung, Bielefeld, Germany}
\email{max.nendel@uni-bielefeld.de}

\renewcommand{\shortauthors}{B. Barraza Mart\'inez et al.}

\date{\today}

\let\thefootnote\relax\footnote{Financial Support through DAAD, COLCIENCIAS via Project 121571250194 and the German Research Foundation via CRC 1283 ``Taming Uncertainty'' is gratefully acknowledged.}

\begin{abstract}
We consider a transmission problem where a structurally damped plate equation is coupled with a damped or undamped wave equation by transmission conditions. We show that exponential stability holds in the damped-damped situation and polynomial stability (but no exponential stability) holds in the damped-undamped case. Additionally, we show that the solutions first defined by the weak formulation, in fact have higher Sobolev space regularity.
\end{abstract}



\maketitle

\section{Introduction}
In this paper, we study a coupled plate-membrane  system, where we assume structural damping for the plate and damping / no damping for the wave equation. More precisely, we consider the following geometric situation: Let $\Omega\subset\R^2$ be a bounded $C^4$-domain with boundary $\Gamma$, and let $\Omega_2\subset\Omega$ be a non-empty bounded $C^4$-domain satisfying $\bar\Omega_2\subset \Omega$. We set $\Omega_1:=\Omega\setminus\bar\Omega_2$ and $I := \partial \Omega_2$. Then $I$ is the interface between $\Omega_1$ and $\Omega_2$ (see Figure~\ref{fig1} for the geometric situation). By $\nu$, we denote the outer unit normal with respect to $\Omega_1$ both on $\Gamma$ and on $I$.

In $\Omega_1\cup \Omega_2$, we consider the plate-membrane (plate-wave) system
\begin{align}
	u_{tt} + \De^2 u - \rho \De u_t 	& = 0\quad \text{in} \, (0,\infty)\times\Om_1, 	\label{eq_om_1}\\
w_{tt} - \De w + \beta w_t			& = 0\quad  \text{in}\, (0,\infty)\times \Om_2,	\label{eq_om_2}
\end{align}
where $\rho\ge 0$ and $\beta\ge 0$ are fixed constants. For $\rho>0$, we have structural damping for the plate equation \eqref{eq_om_1}, whereas the coefficient $\beta\ge 0$ describes the damping (or the absence of damping) for the wave equation \eqref{eq_om_2}.
On the outer boundary $\Gamma$, we impose clamped (Dirichlet) boundary conditions
\begin{equation}
	u =\partial_\nu u 					= 0\quad \text{on}\, (0,\infty)\times \Ga .		\label{eq_bc_1}
\end{equation}

\begin{figure}[ht]
	\begin{center}
	\begin {tikzpicture}[x=3em,y=3em]
		\draw (0,0) ellipse (1.05  and 0.7);
		\draw (0.5, 0.2 ) ellipse (2.12 and 1.58);
		\draw [->,bend left=15]  (-2.2, 0.5) to (-1.6,0.35);
		\draw [->,bend left=15] (1.3, -0.5) to (0.95,-0.3);
		\draw [->, thick] (0.77,0.47) to node[below]{$\quad\nu$} +(-0.5,-0.6);
		\draw [->, thick] (2.58,0.5) to node[below]{$\nu$} +(0.8,0.2);
		\node at (-0.3,0) {$\Omega_2$};
		\node at (-0.3, 1) {$\Omega_1$};
		\node at (-2.5,0.5) {$\Gamma$};
		\node at (1.5,-0.5) {$I$};
	\end{tikzpicture}
	\end{center}
	\caption{The set $\Omega=\Omega_1\cup I \cup\Omega_2$. \label{fig1}}
\end{figure}
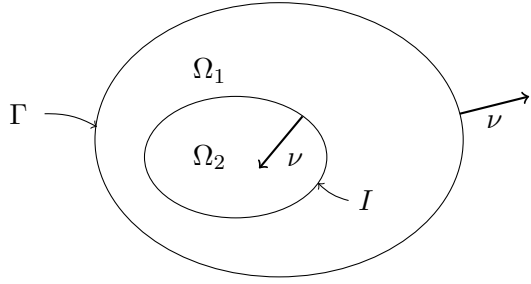

On the interface $I$, we have transmission conditions of the form
\begin{align}
	u																& = w\qquad \,\,\,\, \text{on}\, (0,\infty)\times I, 				\label{eq_tmc_0}\\
 \BB_1 u		 													& = 0\qquad \quad \text{on}\, (0,\infty)\times I, 					\label{eq_tmc_1}\\
 \BB_2 u -\rho \partial_\nu u_t 									& = \md{-}\partial_\nu w\quad  \text{on}\, (0,\infty)\times I 	\label{eq_tmc_2}
\end{align}
with
\[
 \BB_1 u:=\De u+(1-\mu)B_1u\quad \text{and}\quad\BB_2 u:=\partial_\nu\De u+(1-\mu)\partial_\tau B_2u,
\]
where
\[
 B_1u:= - \langle \tau,(\nabla^2 u)\tau \rangle\quad \text{and}\quad B_2u:=\langle \tau , (\nabla^2 u)\nu\rangle.
\]
Here, $\mu\in \big(0,\frac{1}{2}\big)$ is Poisson's ratio and $\tau:=(-\nu_2,\nu_1)^\top$. As we have a coupling of a fourth-order equation with a second-order equation, we have two transmission conditions (\eqref{eq_tmc_0} and \eqref{eq_tmc_2}) and one boundary condition \eqref{eq_tmc_1} on the interface $I$.

Finally, the boundary-transmission problem \eqref{eq_om_1}--\eqref{eq_tmc_2} is endowed with initial conditions of the form
\begin{align}
   u|_{t=0} & = u_0,\; u_t|_{t=0} = u_1 \quad \text{in } \Omega_1, \label{eq_init_1}\\
  w|_{t=0} & = w_0,\; w_t|_{t=0} = w_1\quad\text{in }\Omega_2. \label{eq_init_2}
\end{align}

The aim of the present paper is to investigate well-posedness as well as regularity and stability of the solution of \eqref{eq_om_1}--\eqref{eq_init_2}. Note that we omitted all physical constants for simplicity. Concerning the modelling of plate-membrane systems and more detailed models including physical constants, we refer to, e.g.,  \cite{Arango-Lebedev-Vorovich98}, \cite{Gazizullin-Paimushin16}, and \cite{Hernandez05}.

It is well known that the structurally damped plate equation itself has exponential stability and leads to the generation of an analytic $C_0$-semigroup even in the $L^p$-setting, see \cite{Denk-Schnaubelt15} and the references therein. Due to the hyperbolic structure of the wave equation \eqref{eq_om_2}, $L^p$-theory is not feasible for the coupled system, and we will consider the plate-membrane system in an $L^2$-framework. It is not hard to see the for all $\rho\ge 0$ and $\beta\ge 0$ we have well-posedness, i.e. generation of a $C_0$-semigroup in the corresponding $L^2$-Sobolev spaces (see Theorem~\ref{thm:wellposed} below). The main results of the present paper state that we have exponential stability if both dampings are present ($\rho>0$ and $\beta>0$) but no exponential stability if the wave equation is undamped ($\beta=0$), see Theorems~\ref{3.1} and \ref{3.2}. In the case of a structurally damped plate equation and an undamped wave equation ($\rho>0$ and $\beta=0$) we obtain polynomial stability (Theorem~\ref{thm:pol}).  Moreover, the ``good'' parabolic structure of the damped plate equation implies high elliptic regularity for $u$ and $w$ (Theorem~\ref{4.5}). In particular, the transmission conditions \eqref{eq_tmc_0}--\eqref{eq_tmc_2} hold in the sense of boundary traces.

There is a huge amount of literature on transmission problems for elastic systems, most of them dealing with wave-wave systems. For wave-plate transmission problems, we mention \cite{Hassine16}, where Kelvin-Voigt damping for the plate equation is considered (see also \cite{Hassine16a} for the one-dimensional case). In \cite{Ammari-Nicaise10} exponential stability was obtained for a damped wave / damped plate transmission problem under some geometric condition which leads to a flat interface. This was generalized in \cite{Zhang-Zhang15} to a model with  curved middle surface by virtue of geometric multiplier method. In \cite{Gong-Yang-Zhao17}, stabilization of a damped wave / damped plate system with variable coefficients is studied by means of a Riemannian geometrical approach. For stability of coupled wave-plate systems within the same domain, we mention, e.g., \cite{Liu-Su06}.

Whereas the above mentioned results show exponential stability for many cases of damped-damped systems, this cannot be expected in the damped-undamped situation where we have, from a mathematical point of view, a parabolic-hyperbolic coupled system (see, e.g., \cite{Avalos-Lasiecka-Triggiani16}, \cite{Batty-Paunonen-Seifert16}, \cite{Duyckaerts07} for heat-wave systems).

For transmission problems in (thermo-)viscoelasticity, we mention, e.g., \cite{MR3193931}, \cite{MR3180871}, \cite{MR2807974}, \cite{MR2054600}, and \cite{MunozRivera-Racke17}. In particular, in \cite{MunozRivera-Racke17} polynomial stability for a (thermo-) viscoelastic damped-undamped system with Kelvin-Voigt damping has been shown. The proof is based on an extended version of a characterization of polynomial stability due to Borichev and Tomilov \cite{Borichev-Tomilov10}. It turns out that some arguments in \cite{MunozRivera-Racke17} can be adapted to the plate-wave situation considered in the present paper to show that the system is not exponentially but polynomially stable (Section~5).  We remark that our proof of polynomial stability is based on rather general methods which should be applicable for other transmission problems. However, by this method we do not obtain optimal polynomial rates. The proof of higher regularity (Section~4) uses arguments similar to \cite{Denk-Kammerlander18} where damped plate / undamped plate transmission problems were investigated. In particular, we apply the classical theory of parameter-dependent boundary value problems (see \cite{Agranovich-Vishik64}) to obtain sufficiently good estimates in the damped part.

The structure of the paper is as follows: In Section~2, we define the basic spaces and operators and show the generation of a $C_0$-semigroup of contractions. Exponential stability for $\rho>0$ and $\beta>0$ and non-exponential stability for $\beta=0$ is shown in Section~3,  whereas the proof of higher regularity based on parameter-elliptic theory can be found in Section~4. Finally, polynomial stability for $\rho>0$ and $\beta=0$ is proven in Section~5.

\section{Well-posedness}

We denote by $H^2_{\Ga}(\Om_1)$ the space of all $u\in H^2(\Om_1)$ with $u|_{\Ga}=\partial_\nu u|_{\Ga}=0$. On $H^2_{\Ga}(\Om_1)$ we consider the inner product
\[
 \langle u,v\rangle_{H^2_{\Ga}(\Om_1)}:= \int_{\Om_1}\nabla^2 u:\nabla^2 \overline{v}+\mu [u,\overline{v}]\, {\rm d}x,
\]
where
\[
 \nabla^2 u:\nabla^2 v:= u_{x_1x_1}v_{x_1x_1}+u_{x_2x_2}v_{x_2x_2}+2u_{x_1x_2}v_{x_1x_2}
\]
and
\[
 [u,v]:= u_{x_1x_1}v_{x_2x_2}+u_{x_2x_2}v_{x_1x_1}-2 u_{x_1x_2}v_{x_1x_2}
\]
for all $u,v\in H^2_{\Ga}(\Om_1)$. We thus have that
\[
 \langle u,v\rangle_{H^2_{\Ga}(\Om_1)}=\mu \langle \Delta u,\Delta v\rangle_{L^2(\Om_1)}+(1-\mu) \langle \nabla^2 u,\nabla^2 v\rangle_{L^2(\Om_1)^{4}}
\]
for all $u,v\in H^2_{\Ga}(\Om_1)$. By Poincar\'e's inequality, we have that
\begin{align*}
\Vert u \Vert_{H^2(\Om_1)}^2 & \leq C \big(\Vert \nabla u \Vert_{L^2(\Om_1)^2}^2 + \Vert \nabla^2u \Vert_{L^2(\Om_1)^4}^2 \big)\\
& = C\big(  \Vert u_{x_1} \Vert_{L^2(\Om_1)}^2 + \Vert u_{x_2} \Vert_{L^2(\Om_1)}^2 + \Vert \nabla^2u \Vert_{L^2(\Om_1)^4}^2 \big) \\
& \leq C \big( \Vert \nabla u_{x_1} \Vert_{L^2(\Om_1)^2}^2 + \Vert \nabla u_{x_2}\Vert_{L^2(\Om_1)^2}^2  + \Vert \nabla^2u \Vert_{L^2(\Om_1)^4}^2 \big)\\
& \le C \Vert \nabla^2u \Vert_{L^2(\Om_1)^4}^2
\end{align*}
for $u\in H^2_\Ga(\Om_1)$. Here and in the following, $C$ denotes a generic constant which may change at each appearance.
The above estimate shows that  $\|\cdot \|_{H^2_{\Ga}(\Om_1)}$ is equivalent to the $H^2(\Om_1)$-norm on $H^2_{\Ga}(\Om_1)$. In particular, $\big(H^2_{\Ga}(\Om_1),\langle \cdot ,\cdot \rangle_{H^2_{\Ga}(\Om_1)}\big)$ is a Hilbert space.

We will also use the following result on integration by parts.

\begin{lemma}\label{Lemma-Green-formula-bilaplace-mu}(See \cite{Chueshov-Lasiecka10}, p. 27.)
For $u\in H^4(\Om_1)\cap H^2_\Ga(\Om_1)$ and $v\in H^2_\Ga(\Om_1)$ it holds
\begin{equation}\label{Eq-Green-formula-bilaplace-mu}
\langle \Delta^2u , v \rangle_{L^2(\Om)} = \langle u , v \rangle_{H^2_\Ga(\Om_1)} - \langle \BB_1u , \partial_\nu v \rangle_{L^2(I)} + \langle \BB_2u , v \rangle_{L^2(I)}.
\end{equation}
\end{lemma}
Let
\[
 \HH:=\big\{U = (u_1, v_1, u_2, v_2)^\top \in H^2_{\Ga}(\Om_1)\times L^2(\Om_1) \times H^1(\Om_2)\times L^2(\Om_2) \colon u_1|_{I}=u_2|_{I}\big\}
\]
be endowed with the inner product
\begin{align*}
 \langle U, \widetilde U \rangle_\HH
 	& := \langle u_1,\widetilde u_1\rangle_{H^2_{\Ga}(\Om_1)} + \langle v_1, \widetilde v_1 \rangle_{L^2(\Om_1)} \\
 	& + \langle \nabla u_2,\nabla \widetilde u_2 \rangle_{L^2(\Om_2)^2} + \langle v_2 ,\widetilde v_2 \rangle_{L^2(\Om_2)}
\end{align*}
for $U, \widetilde U \in \HH$.
Then $(\HH,\langle\cdot , \cdot\rangle_\HH)$ is a Hilbert space. Note that we can omit the term $\langle u_2, \tilde u_2\rangle_{L^2(\Omega_2)}$ by Poincar\'e's inequality  as $u_1\chi_{\Omega_1} + u_2\chi_{\Omega_2} \in H^1_0(\Omega)$. Here, $\chi_{\Omega_j}$ stands for the characteristic function of $\Omega_j$.


We introduce the operator matrix $A$ given by
\[
 A:=\left(\begin{array}{cccc}
              0 & 1  & 0 & 0\\
              -\De^2 & \rho \De & 0 & 0\\
              0 & 0 & 0 & 1\\
              0 & 0 & \De & -\beta
             \end{array}
       \right).
\]
By \eqref{Eq-Green-formula-bilaplace-mu},
we have that
\begin{align*}
	\langle AU, \widetilde U \rangle_\HH
		& = \langle v_1, \widetilde u_1 \rangle_{H^2_{\Ga}(\Omega_1)} - \langle \Delta^2 u_1 - \rho \Delta v_1, \widetilde v_1 \rangle_{L^2(\Omega_1)} \\
			& \quad + \langle \nabla v_2, \nabla \widetilde u_2 \rangle_{L^2(\Omega_2)^{\md{2}}}	+ \langle \Delta u_2 - \beta v_2, \widetilde v_2 \rangle_{L^2(\Omega_2)} \\
		& = \langle v_1, \widetilde u_1 \rangle_{H^2_{\Ga}(\Omega_1)} + \langle \nabla v_2, \nabla \widetilde u_2 \rangle_{L^2(\Omega_2)^{\md{2}}}
			- \langle u_1, \widetilde v_1 \rangle_{H^2_{\Ga}(\Omega_1)} \\
			& \quad - \langle \BB_2 u_1, \widetilde v_1 \rangle_{L^2(\partial \Omega_1)}
				+ \langle \BB_1 u_1, \partial_\nu \widetilde v_1 \rangle_{L^2(\partial \Omega_1)} \\
			& \quad - \rho \langle \nabla v_1, \nabla \widetilde v_1 \rangle_{L^2(\Omega_1)^{\md{2}}} + \rho \langle \partial_\nu v_1, \widetilde v_1 \rangle_{L^2(\partial \Omega_1)} \\
			& \quad - \langle \nabla u_2, \nabla \widetilde v_2 \rangle_{L^2(\Omega_2)^{\md{2}}} - \beta \langle v_2, \widetilde v_2 \rangle_{L^2(\Omega_2)}
				\md{-} \langle \partial_\nu u_2, \widetilde v_2 \rangle_{L^2(I)}
\end{align*}
for all sufficiently smooth $U, \widetilde U$. This leads us to the following interpretation of the transmssion conditions \eqref{eq_tmc_1} and \eqref{eq_tmc_2}:
we say that $U$ \emph{satisfies the transmission conditions \eqref{eq_tmc_1} and \eqref{eq_tmc_2} weakly} if the equality
\begin{multline}\label{weak-tm}
	\langle AU, \Phi \rangle_\HH
		= \langle v_1, \varphi_1 \rangle_{H^2_{\Ga}(\Omega_1)} + \langle \nabla v_2, \nabla \varphi_2 \rangle_{L^2(\Omega_2)^{\md{2}}}
				- \langle u_1, \psi_1 \rangle_{H^2_{\Ga}(\Omega_1)} \\
			 \quad - \rho \langle \nabla v_1, \nabla \psi_1 \rangle_{L^2(\Omega_1)^{\md{2}}}
				- \langle \nabla u_2, \nabla \psi_2 \rangle_{L^2(\Omega_2)^{\md{2}}} - \beta \langle v_2, \psi_2 \rangle_{L^2(\Omega_2)}
\end{multline}
holds true for all
$\Phi = (\varphi_1, \psi_1, \varphi_2, \psi_2)^\top \in H^2_{\Ga}(\Omega_1) \times H^2_{\Ga}(\Omega_1) \times H^1(\Omega_2) \times H^1(\Omega_2)$
satisfying $\varphi_1 = \varphi_2$ \md{and $\psi_1 = \psi_2\ $} on $I.$ \\

Now, we consider the linear operator $\AA\colon D(\AA)\subset \HH\to \HH, \, U \mapsto AU$ with
\begin{align*}
	D(\AA) := \big\{ U \in \HH : & \,
			v_1 \in H^2_{\Ga}(\Omega_1), v_2 \in H^1(\Omega_2), \Delta^2 u_1 \in L^2(\Omega_1), \Delta u_2 \in L^2(\Omega_2), \\
		 & v_1 = v_2 \text{ on } I \text{ and } \eqref{eq_tmc_1}, \eqref{eq_tmc_2} \text{ are weakly satisfied} \big\}.
\end{align*}

As
\begin{equation}\label{dissieq}
	\md{\Re}\,\langle \AA U,U\rangle_\HH = -\rho \|\nabla v_1 \|_{L^2(\Om_1)^{\md{2}}}^2 - \beta \|\md{v_2} \|_{L^2(\Om_2)}^2\leq 0
\end{equation}
for all $U\in D(\AA)$, the operator $\AA$ is dissipative. The same argument shows that for any smooth solution $(u, w)$ of \eqref{eq_om_1}-\eqref{eq_tmc_2}, the energy
\begin{align*}
	E(t)	& := \md{\frac{1}{2}}\int_{\Omega_1} \mu \vert \Delta u(t) \vert^2 + (1-\mu) \vert \nabla^2 u(t) \vert^2 + \vert u_t(t) \vert^2 \, dx \\
			& \quad + \md{\frac{1}{2}}\int_{\Omega_2} \vert \nabla w(t) \vert^2 + \vert w_t(t) \vert^2 \, dx
\end{align*}
is decreasing and the dissipation is caused by the damping both in $\Omega_1$ and $\Omega_2$. Moreover, the system is still dissipative if only one of the damping terms is active
($\rho + \beta  > 0$)
and the system is conservative if there is no damping at all ($\rho = \beta = 0)$.

In what follows, we show that the system \eqref{eq_om_1}-\eqref{eq_tmc_2} is well-posed for any choice of $\rho \geq 0$ and $\beta \geq 0.$

\begin{theorem}\label{thm:wellposed}
	The operator $\AA\colon \HH \supset D(\AA) \to \HH$ generates a strongly continuous semigroup $(S(t))_{t\geq 0}$ of contractions on $\HH$.
\end{theorem}

\begin{proof}
First, we show that $1-\AA$ is surjective.
	Let $F = (f_1, g_1, f_2, g_2)^\top \in \HH.$ We need to show that there exists a $U = (u_1, v_1, u_2, v_2)^\top \in D(\AA)$ such that $(1-\AA)U = F,$ i.e.
	\begin{align*}
		u_1 - v_1								& = f_1, \\
		v_1 + \Delta^2 u_1 - \rho \Delta v_1	& = g_1, \\
		u_2 - v_2								& = f_2, \\
		v_2 - \Delta u_2 + \beta v_2			& = g_2.
	\end{align*}
	Plugging in $v_i = u_i - f_i$ for $i = 1, 2$, we have to solve
	\begin{align}
		u_1 + \Delta^2 u_1 - \rho \Delta u_1	& = f_1 + g_1 - \rho \Delta f_1,	\label{eq_surjective_1} \\
		u_2 - \Delta u_2 + \beta u_2			& = f_2 + g_2 + \beta f_2.			\label{eq_surjective_2}
	\end{align}
	Motivated by the notion of the weak transmission conditions, we introduce the space
	\begin{align*}
		\VV := \{ u=(u_1, u_2)^\top \in H^2_{\Ga}(\Omega_1) \times H^1(\Omega_2) : u_1 = u_2 \text{ on } I \}.
	\end{align*}
	Endowed with the scalar product
	\begin{align*}
		\langle u, \widetilde u \rangle_\VV = \langle u_1, \widetilde u_1 \rangle_{H^2_{\Ga}(\Omega_1)} + \langle \nabla u_2, \nabla \widetilde u_2 \rangle_{L^2(\Omega_2)^2}
		\qquad (u, \widetilde u \in \VV),
	\end{align*}
	$(\VV, \langle \cdot, \cdot \rangle_\VV)$ becomes a Hilbert space. \\
	In order to solve \eqref{eq_surjective_1}, \eqref{eq_surjective_2}, we will use the theorem of Lax-Milgram in the Hilbert space $\VV.$ Let $b \colon \VV \times \VV \to \R$
	be defined by
	\begin{align*}
		b\left( u, \varphi \right)
			& := \langle u_1, \varphi_1 \rangle_{L^2(\Omega_1)} + \langle u_1, \varphi_1 \rangle_{H^2_{\Ga}(\Omega_1)}
					+ \rho \langle \nabla u_1, \nabla \varphi_1 \rangle_{L^2(\Omega_1)^2} \\
				& \quad + (1+\beta) \langle u_2, \varphi_2 \rangle_{L^2(\Omega_2)} + \langle \nabla u_2, \nabla \varphi_2 \rangle_{L^2(\Omega_2)^2}.
	\end{align*}
	Obviously, $b$ is bilinear and continuous. Since
	\begin{align*}
		b(u, u)
			& = \Vert u_1 \Vert_{L^2(\Omega_1)}^2 + \Vert u_1 \Vert_{H^2_{\Ga}(\Omega_1)}^2 + \rho \Vert \nabla u_1 \Vert_{L^2(\Omega_1)^2}^2 \\
				& \quad + (1+\beta) \Vert u_2 \Vert_{L^2(\Omega_2)}^2 + \Vert \nabla u_2 \Vert_{L^2(\Omega_2)^2}^2 \\
			& \geq \Vert u_1 \Vert_{H^2_{\Ga}(\Omega_1)}^2 + \Vert \nabla u_2 \Vert_{L^2(\Omega_2)^2}^2
	\end{align*}
	holds for all $u \in \VV,$ the bilinear form $b$ is coercive on $\VV.$ Hence, there exists a unique $u \in \VV$ satisfying
	\begin{align}\label{eq_lax_milgram}
		b(u,  \varphi ) = \Lambda (\varphi )
	\end{align}
	for all $\varphi \in \mathcal{V}$, where the linear functional $\Lambda \colon \VV \to \R$ is given by
	\begin{align*}
		\Lambda(\varphi )
			& := \langle f_1 + g_1, \varphi_1 \rangle_{L^2(\Omega_1)} + \rho \langle \nabla f_1, \nabla \varphi_1 \rangle_{L^2(\Omega_1)^2} \\
				& \quad + \langle g_2 + (1 + \beta) f_2, \varphi_2 \rangle_{L^2(\Omega_2)}.
	\end{align*}
	Note that for $\varphi_1 \in C_0^\infty(\Omega_1)$ we have
	\begin{align*}
		\langle u_1, \varphi_1 \rangle_{H^2_{\Ga}(\Omega_1)}
			& = \langle u_1, \Delta^2 \varphi_1 \rangle_{L^2(\Omega_1)} - \langle u_1, \BB_2 \varphi_1 \rangle_{L^2(\partial \Omega_1)}
					+ \langle \partial_\nu u_1, \BB_1 \varphi_1 \rangle_{L^2(\partial \Omega_1)} \\
			& = \langle \Delta u_1, \Delta \varphi_1 \rangle_{L^2(\Omega_1)}.
	\end{align*}
	In particular, for any $(\varphi_1, \varphi_2) \in C_0^\infty(\Omega_1) \times C_0^\infty(\Omega_2) \subset \VV$, we have that \eqref{eq_surjective_1} and
	\eqref{eq_surjective_2} are satisfied in $L^2(\Omega_1)$ and $L^2(\Omega_2)$, respectively. This implies that $\Delta^2 u_1 \in L^2(\Omega_1)$ and
	$\Delta u_2 \in L^2(\Omega_2).$ We set
	\begin{align*}
		U := 	\begin{pmatrix}
					u_1 \\
					u_1 - f_1 \\
					u_2 \\
					u_2 - f_2
				\end{pmatrix} \in \HH.
	\end{align*}
	Finally, using \eqref{eq_surjective_1}, \eqref{eq_surjective_2} and \eqref{eq_lax_milgram}, we calculate
	\begin{align*}
		\langle \AA U, \Phi \rangle_{\HH}
			& = \langle v_1, \varphi_1 \rangle_{H^2_{\Ga}(\Omega_1)} - \langle \Delta^2 u_1 - \rho \Delta v_1, \psi_1 \rangle_{L^2(\Omega_1)} \\
				& \quad + \langle \nabla v_2, \nabla \varphi_2 \rangle_{L^2(\Omega_2)^2} + \langle \Delta u_2 - \beta v_2, \psi_2 \rangle_{L^2(\Omega_2)} \\
			& =  \langle v_1, \varphi_1 \rangle_{H^2_{\Ga}(\Omega_1)} - \langle g_1 + f_1, \psi_1 \rangle_{L^2(\Omega_1)} + \langle u_1, \psi_1 \rangle_{L^2(\Omega_1)} \\
				& \quad \quad + \langle \nabla v_2, \nabla \varphi_2 \rangle_{L^2(\Omega_2)^2} - \langle f_2 + g_2, \psi_2 \rangle_{L^2(\Omega_2)} \\
			& = \langle v_1, \varphi_1 \rangle_{H^2_{\Ga}(\Omega_1)} + \langle \nabla v_2, \nabla \varphi_2 \rangle_{L^2(\Omega_2)^2} \\
				& \quad - \rho \langle \nabla v_1, \nabla \psi_1 \rangle_{L^2(\Omega_1)^2} - \beta \langle v_2, \psi_2 \rangle_{L^2(\Omega_2)} \\
				& \quad  - \langle u_1, \psi_1 \rangle_{H^2_{\Ga}(\Omega_1)} - \langle \nabla u_2, \nabla \psi_2 \rangle_{L^2(\Omega_2)^2}
	\end{align*}
	for any $\Phi = (\varphi_1, \psi_1, \varphi_2, \psi_2)^\top \in H^2_{\Ga}(\Omega_1) \times H^2_{\Ga}(\Omega_1) \times H^1(\Omega_2) \times H^1(\Omega_2)$
	satisfying $\varphi_1 = \varphi_2$ \md{and $\psi_1 = \psi_2$}\  on $I.$
Therefore, $U$ satisfies the transmission conditions weakly. Hence, $U \in D(\AA)$ and $(1-\AA)U = F.$

As $\AA$ is dissipative and $1 - \AA$ is surjective,   $\AA$ generates a $C_0$-semigroup of contractions by the  Lumer-Phillips Theorem.
\end{proof}

\begin{remark}
In the same way as in the previous proof, one can show that the operator
$\AA$ is continuously invertible, i.e. $0$ belongs to the resolvent set
$\rho(\AA)$. To show this, we now have to consider%
\begin{align}
\Delta^{2}u_{1}  & =g_{1}-\rho\Delta f_{1}\text{,}\label{ec1 ZeroInResolvent}%
\\
-\Delta u_{2}  & =g_{2}+\beta f_{2}\label{ec2 ZeroInResolvent}%
\end{align}
instead of \eqref{eq_surjective_1} and \eqref{eq_surjective_2}. The sesquilinear form $B$ and the functional
$\Lambda$ are now defined by $B\left(  u,\varphi\right)  :=\left\langle
u,\varphi\right\rangle _{\mathcal{V}}$ and%
\[
\Lambda\left(  \varphi\right)  :=\left\langle g_{1},\varphi_{1}\right\rangle
_{L^{2}\left(  \Omega_{1}\right)  }+\rho\left\langle \nabla f_{1}%
,\nabla\varphi_{1}\right\rangle _{L^{2}\left(  \Omega_{1}\right)  ^{2}%
}+\left\langle g_{2},\varphi_{2}\right\rangle _{L^{2}\left(  \Omega
_{2}\right)  }+\beta\left\langle f_{2},\varphi_{2}\right\rangle _{L^{2}\left(
\Omega_{2}\right)  }%
\]
for $u=\left(  u_{1},u_{2}\right)  $, $\varphi=\left(  \varphi_{1},\varphi
_{2}\right)  \in\mathcal{V}$. The Riesz Representation Theorem implies that
there exists a unique solution $u=\left(  u_{1},u_{2}\right)  \in\mathcal{V}$
satisfying%
\begin{equation}
B\left(  u,\varphi\right)  =\Lambda\left(  \varphi\right)  \qquad\text{for all
}\varphi\in\mathcal{V}.\label{ec3 ZeroInResolvent}%
\end{equation}
In particular, choosing $\left(  \varphi_{1},\varphi_{2}\right)  \in
C_{0}^{\infty}\left(  \Omega_{1}\right)  \times C_{0}^{\infty}\left(
\Omega_{2}\right)  \subset\mathcal{V}$ we see that (\ref{ec1 ZeroInResolvent})
and (\ref{ec2 ZeroInResolvent}) hold in the sense of distributions in
$\Omega_{1}$ and $\Omega_{2},$ respectively. As the right-hand side of
(\ref{ec1 ZeroInResolvent}) belongs to $L^{2}\left(  \Omega_{1}\right)  $, the
same holds for the left-hand side, i.e. $\Delta^{2}u_{1}\in L^{2}\left(
\Omega_{1}\right)  $. In the same way, we see that (\ref{ec2 ZeroInResolvent})
holds as equality in $L^{2}\left(  \Omega_{2}\right)  $ and therefore $\Delta
u_{2}\in L^{2}\left(  \Omega_{2}\right)  $. Now, set%
\begin{equation}
v_{i}:=-f_{i}\text{ \ \ \ for }i=1,2.\label{ec4 ZeroInResolvent}%
\end{equation}
Then $U:=\left(  u_{1},v_{1},u_{2},v_{2}\right)  ^{\top}\in\HH$ and%
\begin{equation}
\Delta^{2}u_{1}-\rho\Delta v_{1}=g_{1}\text{,}\qquad-\Delta u_{2}+\beta
v_{2}=g_{2}.\label{ec5 ZeroInResolvent}%
\end{equation}
In the same way as in the proof of Theorem~\ref{thm:wellposed}, one sees that $U$ satisfies the transmission conditions weakly.
Therefore $U$ beolongs to $D(\AA)$ and satisfies $-\AA U=F$.

On other hand, if $\tilde{U}\in D(\AA)$ solves $-\AA\tilde{U}=F$, then $B\left(  \tilde{u},\varphi\right)  = \Lambda\left(\varphi\right)$
holds for all $\varphi\in\mathcal{V}$ due to the definition
of $D(\AA)$ and the weak transmission conditions. Therefore
$U=\tilde{U}$, and $\AA$ is a bijection. Since $\AA$ is the
generator of a $C_{0}-$semigroup by Theorem~\ref{thm:wellposed}, $\AA$ is closed and hence
$0\in\rho\left(\AA\right)  $.
\end{remark}

\section{Results on exponential stability}\label{expstab}

In this section, we study exponential stability of the semigroup $(S(t))_{t\ge 0}$ generated by $\AA$. First, we consider the case where we have damping in both sub-domains, i.e., $\rho>0$ and $\beta>0$. It is no surprise that in this case exponential stability holds.

\begin{theorem}
  \label{3.1}
  Let $\rho>0$ and $\beta > 0$. Then the semigroup $(S(t))_{t\ge 0}$ is exponentially stable, i.e., for any $U_0\in D(\AA)$ and $U(t) := S(t) U_0\;(t\ge 0)$ we have $E(t)\le Ce^{-\kappa t}E(0)$ with positive constants $C$ and $\kappa$, where $E(t) := \frac 12 \|U(t)\|_{\HH}^2$.
\end{theorem}

\begin{proof}
  Let $U(t) = (u_1(t),v_1(t),u_2(t),v_2(t))^\top = S(t)U_0$ with $U_0\in D(\AA)$. For the energy $E(t)$ we obtain
  \begin{equation}
    \label{3-1}
    E'(t) = \Re\langle \AA U(t), U(t)\rangle_{\HH} = -\rho \|\nabla v_1(t)\|_{L^2(\Omega_1)^2}^2 - \beta \|v_2(t)\|_{L^2(\Omega_2)}^2.
  \end{equation}
  We define $F(t) := \langle u_1(t),v_1(t)\rangle_{L^2(\Omega_1)} + \langle u_2(t), v_2(t)\rangle _{L^2(\Omega_2)}$ for $t\ge 0$. Then
  \[ |F(t)| \le \tfrac12 \Big( \|u_1(t)\|_{L^2(\Omega_1)}^2 + \|v_1(t)\|_{L^2(\Omega_1)}^2 +\|u_2(t)\|_{L^2(\Omega_2)}^2 + \|v_2(t)\|_{L^2(\Omega_2)}^2\Big).\]
  By definition of $\HH$, we have $u_1(t)=u_2(t)$ on the interface $I$, and therefore the function $u(t):= u_1(t)\chi_{\Omega_1} + u_2(t)\chi_{\Omega_2}$
  belongs to $H^1_0(\Omega)$ for all $t \geq 0$. An application of  Poincar\'{e}'s     inequality yields
  \begin{align*}
    \|u_1(t)\|_{L^2(\Omega_1)}^2 + \|u_2(t)\|_{L^2(\Omega_2)}^2 & = \|u(t)\|_{L^2(\Omega)} ^2 \le
    C\|\nabla u(t)\|_{L^2(\Omega)^2}^2 \\
    & =C\Big( \|\nabla u_1(t)\|_{L^2(\Omega_1)^2}^2 + \|\nabla u_2(t)\|_{L^2(\Omega_2)^2}^2\Big)\\
    & \le C\Big( \|u_1(t)\|_{H^2_\Gamma(\Om_1)}^2 + \|\nabla u_2(t)\|_{L^2(\Omega_2)^2}^2\Big).
  \end{align*}
  Therefore, for some constant $c_1>0$ we get
  \begin{equation}
    \label{3-2}
    |F(t)|\le \frac{c_1}2 \|U(t)\|_{\HH}^2 = c_1 E(t).
  \end{equation}
  Using $U'(t) = \AA U(t)$, we obtain
  \begin{align*}
    F'(t) & = \langle u_1'(t),v_1(t)\rangle_{L^2(\Omega_1)} + \langle u_1(t),v_1'(t)\rangle_{L^2(\Omega_1)} \\
    & \quad + \langle u_2'(t), v_2(t)\rangle_{L^2(\Omega_2)} + \langle u_2(t),v_2'(t)\rangle_{L^2(\Omega_2)}\\
    & = \|v_1(t)\|_{L^2(\Omega_1)}^2 - \langle u_1(t),\Delta^2 u_1(t)-\rho\Delta v_1(t) \rangle_{L^2(\Omega_1)} \\
    & \quad+ \|v_2(t)\|_{L^2(\Omega_2)}^2 + \langle u_2(t),\Delta u_2(t)-\beta v_2(t) \rangle_{L^2(\Omega_2)}.
  \end{align*}
  Now we use the fact that $U(t)\in D(\AA)$ and take $\Phi:= (0, u_1(t), 0, u_2(t))^\top$ in the weak transmission conditions \eqref{weak-tm}. We obtain
  \begin{align*}
    F'(t) & = \|v_1(t)\|_{L^2(\Omega_1)}^2 + \|v_2(t)\|_{L^2(\Omega_2)}^2 + \langle \Phi, \AA U(t)\rangle_{\HH} \\
    & = \|v_1(t)\|_{L^2(\Omega_1)}^2 + \|v_2(t)\|_{L^2(\Omega_2)}^2 - \|u_1(t)\|_{H^2_\Gamma(\Om_1)}^2 - \|\nabla u_2(t)\|_{L^2(\Omega_2)^2}^2 \\
    & \quad -\rho \langle \nabla u_1(t), \nabla v_1(t)\rangle_{L^2(\Omega_1)^2} - \beta \langle u_2(t), v_2(t)\rangle_{L^2(\Omega_2)}.
  \end{align*}
  By Young's inequality and Poincar\'{e}'s inequality in $\Omega_1$, for every $\delta>0$ there exists a $C_\delta>0$ such that
  \begin{align}
  -\rho\langle \nabla u_1(t),\nabla v_1(t)\rangle_{L^2(\Omega_1)^2} & \le \rho\delta \|\nabla u_1(t)\|_{L^2(\Omega_1)^2}^2 + \rho C_\delta \|\nabla v_1(t)\|_{L^2(\Omega_1)^2}^2 \nonumber\\
  & \le c_2 \rho\delta \|\nabla^2 u_1(t)\|_{L^2(\Omega_1)^2}^2 + \rho C_\delta \|\nabla v_1(t)\|_{L^2(\Omega_1)^2}^2 \nonumber\\
  & \le c_3 \rho\delta \|u_1(t)\|_{H^2_\Gamma(\Om_1)}^2 + \rho C_\delta \|\nabla v_1(t)\|_{L^2(\Omega_1)^2}^2.\label{2-3}
  \end{align}
  In the same way, \md{using Poincar\'{e}'s inequality in $\Omega$},
  \begin{align}\label{2-4}
  \begin{split}
    -\beta \langle u_2(t),v_2(t)\rangle_{L^2(\Omega_2)} & \le \beta \delta \|u_2(t)\|_{L^2(\Omega_2)}^2 + \beta C_\delta \|v_2(t)\|_{L^2(\Omega_2)}^2\\
    & \md{\le \tilde{c_3}\beta\delta\big( \|u_1(t)\|_{H^2_\Gamma(\Om_1)}^2 + \| \nabla u_2(t) \|_{L^2(\Om_2)^2}^2\big)}\\
    & \qquad \md{+ \beta C_\delta \|v_2(t)\|_{L^2(\Omega_2)}^2 }.
    \end{split}
  \end{align}
  Choosing $\delta$ small enough such that \md{$(c_3\rho + \tilde{c_3}\beta)\delta\le\frac12$}, we get from \eqref{2-3} and \eqref{2-4} (again using Poincar\'{e}'s inequality for $v_1(t)$ in $\Omega_1$)
  \begin{align}
    F'(t) & \le c_4 \big( \|\nabla v_1(t)\|_{L^2(\Omega_1)^2}^2 + \|v_2(t)\|_{L^2(\Omega_2)}^2\big) \nonumber\\
    & \quad - \tfrac12 \big( \|u_1(t)\|_{H^2_\Gamma(\Om_1)}^2 + \|\nabla u_2(t)\|_{L^2(\Omega_2)^2}^2\big). \label{2-5}
  \end{align}
  Now let $L(t) := c_5 E(t) + F(t)$, where the constant $c_5$ satisfies $c_5\ge2 c_1$ and $\min\{\rho,\beta\} c_5 \ge c_4+\frac 12$. By \eqref{3-1} and \eqref{2-5} we see that
  \begin{equation}\label{2-6}
  \begin{aligned}
   L'(t) & \le -\frac12\Big( \|u_1(t)\|_{H^2_\Gamma(\Om_1)}^2 + \|\nabla v_1(t)\|_{L^2(\Omega_1)^2}^2 \\
   & \quad + \|\nabla u_2(t)\|_{L^2(\Omega_2)^2}^2 + \|v_2(t)\|_{L^2(\Omega_2)}^2\Big)\\
   & \le - C E(t).
   \end{aligned}
   \end{equation}
 As $|F(t)|\le c_1 E(t) \le \frac{c_5}2 E(t)$, we obtain
 \[ \frac{c_5}2 E(t) \le L(t) \le \frac{3c_5}2 E(t).\]
 Therefore, \eqref{2-6} yields $L'(t) \le -\kappa L(t)$ with some positive constant $\kappa$. By Gronwall's lemma, $L(t) \le e^{-\kappa t} L(0)$ which yields
 \[ E(t) \le C L(t) \le C e^{-\kappa t} L(0) \le C e^{-\kappa t} E(0).\]

 \vspace*{-1.5em}
\end{proof}

Now let us consider the case where the membrane is not damped, i.e., $\beta=0$. In this situation, we show that the system is not exponentially stable,  no matter if $\rh >0$ or $\rh=0$. The proof of the following theorem follows an idea of \cite[Theorem 3.5]{MunozRivera-Racke17}.

\begin{theorem}\label{3.2}
For $\beta=0$ and $\rho\ge 0$, the system is not exponentially stable.
\end{theorem}

\begin{proof}
 We consider the closed subspace
 \[
  \HH_0:=\{0\}\times \{0\}\times H^1_0(\Om_2)\times L^2(\Om_2)
 \]
 of $\HH$. On $\HH_0$ we consider the $C_0$-semigroup $(\widetilde S(t))_{t\geq 0}$ 
with the generator
 \[
  \widetilde \AA\colon \HH_0\supset D(\widetilde \AA)\to \HH_0,\quad U\mapsto \left(\begin{array}{cccc}
  	1& 0 & 0 & 0\\
    0 & 1 & 0 & 0\\
    0 & 0 & 0 & 1\\
    0 & 0 & \De & 0     \end{array}\right)U
 \]
 where $D(\widetilde \AA):=\{0\}\times \{0\}\times (H^2(\Om_2)\cap H_0^1(\Om_2))\times H^1_0(\Om_2)$. In the sequel, we will show that $S(t)-\widetilde S(t)\colon \HH_0\to \HH$ is compact.
For $U_0 \in \HH_0$, we consider
\[
 	E(t):=\frac{1}{2}\big\|S(t)U_0-\widetilde S(t)U_0\big\|_{\HH}^2
\]
for $t\geq 0$. Then, we denote by $(u,u_t,w,w_t)^\top:=S(t)U_0$ the solution of the transmission problem \eqref{eq_om_1}--\eqref{eq_tmc_2}  and
$(0,0,\widetilde w,\widetilde w_t)^\top:=\widetilde S(t)U_0$ the solution of the wave equation in $\Omega_2$ with homogeneous Dirichlet boundary conditions. Then $z:=w-\widetilde w$ solves the wave equation $z_{tt}-\De z=0$ in $\Om_2$ with $z|_I=w|_I=u|_I$. Therefore, applying the weak transmission conditions to $$\langle AS(t)U_0, S(t)U_0-\widetilde{S}(t)U_0\rangle_{\HH}$$ and using integration by parts for $\langle \Delta \widetilde w(t), z_t(t)\rangle_{L^2(\Om_2)}$, we obtain
\begin{align*}
 	E'(t)
 		& = \Re\Big(\langle u(t),u_t(t)\rangle_{H^2_\Ga(\Om_1)}+\langle u_t(t),u_{tt}(t)\rangle_{L^2(\Om_1)} \\
 			& \qquad \quad + \langle\nabla z(t),\nabla z_t(t) \rangle_{L^2(\Om_2)^2}+\langle z_t(t),z_{tt}(t)\rangle_{L^2(\Om_2)}\Big)\\
  		& = \Re\Big(\langle u(t),u_t(t)\rangle_{H^2_\Ga(\Om_1)}
  			+ \langle u_t(t),-\De^2 u(t) +\rh \De u_t(t)\rangle_{L^2(\Om_1)} \\
  			& \qquad \quad -\langle\partial_\nu z(t), z_t(t) \rangle_{L^2(I)} +\langle \BB_2 u(t)+\rh \partial_\nu u_t(t),u_t(t)\rangle_{L^2(I)} \Big)\\
  		& = -\rh\|\nabla u_t(t)\|_{L^2(\Om_1)^2}^2+ \Re(\langle\partial_\nu \widetilde w(t), u_t(t) \rangle_{L^2(I)}).
 \end{align*}
 This implies that
 \begin{equation}\label{notexp1}
  E(t)+\int_0^t \rh\|\nabla u_t(s)\|_{L^2(\Om_1)^2}^2\, {\rm d}s = \int_0^t\Re( \langle\partial_\nu \widetilde w(s), u_t(s)  \rangle_{L^2(I)})\, {\rm d}s.
 \end{equation}
 Now, let $(U_0^k)_{k\in \N}\subset \HH_0$ be a bounded sequence. We define $\widetilde w^k$ and $u^k$  as $\widetilde w$ and $u$ but with $U_0$ being replaced by $U^k_0$ for $k\in \N$. Then, as the sequence $\big(\partial_\nu \widetilde w^k\big)_{k\in \N}\subset L^2\big([0,t];L^2(I)\big)$ is uniformly bounded, there exists a subsequence of $(\widetilde w_k)_{k\in \N}$ which will again be denoted by $(\widetilde w_k)_{k\in \N}$ such that $\big(\partial_\nu \widetilde w^k\big)_{k\in \N}$ converges weakly in $L^2\big([0,t];L^2(I)\big)$. Moreover, the sequences $(u_t^k)_{k\in \N}\subset L^2\big([0,t];H^2(\Om_1)\big)$ and $(u_{tt}^k)_{k\in \N}\subset L^2\big([0,t];L^2(\Om_1)\big)$ are both uniformly bounded. By the Aubin-Lions Lemma, there exists a subsequence of $(u^k)_{k\in \N}$, which will again be denoted by $(u^k)_{k\in \N}$ such that $(u^k_t)_{k\in \N}\subset L^2\big([0,t];H^1(\Om_1)\big)$ converges. As the trace
 \[
  H^1(\Om_1)\to L^2(I),\quad v\mapsto v|_I
 \]
 is continuous, we obtain that $(u_t^k)_{k\in \N}\subset L^2\big([0,t];L^2(I)\big)$ is convergent.
 For $k,l\in \N$ we now denote by
 \[
  E^{kl}(t):=\frac{1}{2}\big\|S(t)(U_0^k-U^l_0)-\widetilde S(t)(U_0^k-U^l_0)\big\|_{\HH}^2.
 \]
 Then, by \eqref{notexp1} we get that
 \[
  E^{kl}(t)\leq \int_0^t\big\langle\partial_\nu \widetilde w^{kl}(s), u^{kl}_t(s) \big\rangle_{L^2(I)}\, {\rm d}s=\big\langle\partial_\nu \widetilde w^{kl}, u^{kl}_t\big\rangle_{L^2\big([0,t];L^2(I)\big)}\to 0
 \]
 as $k,l\to \infty$, where $\widetilde w^{kl}$ and $u^{kl}$ are defined as $\widetilde w$ and $u$ but with $U_0$ being replaced by $U^k_0-U_0^l$ for $k,l\in \N$.
 Therefore, $((S(t)-\tilde S(t))U_0^k)_{k\in\N}$ is a Cauchy sequence in $\HH$ and thus convergent. This shows the compactness of $S(t)-\tilde S(t)\colon \HH_0\to \HH$.
 As $\tilde S(t)$ is the semigroup related to the wave equation, its essential spectral radius equals 1. An application of \cite[Theorem~3.3]{MunozRivera-Racke17} gives that the essential spectral radius of $S(t)$ equals 1, too, and thus $(S(t))_{t\ge 0}$ is not exponentially stable.
\end{proof}

\section{Higher regularity}

In this section, we show that the functions in the domain of $\AA$ have higher regularity, which implies that the transmission conditions hold in the strong sense of traces. For this, we need some results from the theory of parameter-elliptic boundary value problems developed in the 1960's (\cite{Agranovich-Vishik64}, see also \cite{Agranovich-Denk-Faierman97}). Let $\Omega\subset\R^2$ be a domain, and let $A(D) = \sum_{|\alpha|\le 2m} a_\alpha \partial^\alpha$ be a linear differential operator in $\Omega$ of order $2m$. Then $A(D)$ is called parameter-elliptic if  the principal symbol $A(i\xi) := \sum_{|\alpha|=2m} a_\alpha (i\xi)^\alpha$ satisfies
\[ \lambda-A(i\xi) \not=0\quad (\Re\lambda\ge 0,\, \xi\in\R^2,\, (\lambda,\xi)\not=0).\]
Let  $B_1(D),\dots, B_m(D)$ be linear boundary operators on $\partial\Omega$ of the form $B_j(D) = \sum_{|\beta|\le m_j} b_{j\beta} \partial^\beta$ of order $m_j<2m$ with principal symbols $B_j(i\xi) := \sum_{|\beta|=m_j} b_{j\beta}(i\xi)^\beta$. Then we say that the boundary value problem is parameter-elliptic if $A(D)$ is parameter-elliptic and  if the following Shapiro-Lopatinskii condition holds:

\textbf{(SL)} Let $x_0\in\partial\Omega$, and rewrite the boundary value problem $(A(D),$ $B_1(D),\dots, B_m(D))$ in  the coordinate system associated with $x_0$, which is obtained from the original one by a rotation after which the positive $x_2$-axis has the direction of the interior
normal to $\partial \Omega$ at $x_0$. Then the trivial solution $w=0$ is the only stable solution of the ordinary differential equation on the half-line
\begin{align*}
  \md{( \lambda -} A(i\xi_1, \partial_2 )\md{)} w(x_2) & = 0 \quad (x_2\in (0,\infty)),\\
  B_j(i\xi_1,\partial_2) w(0) & = 0 \quad (j=1,\dots,m)
\end{align*}
 for all $\xi_1\in\R$ and $\Re\lambda\ge 0$ with $(\xi_1,\lambda)\not=0$.

It was shown in \cite{Agranovich-Vishik64} that the operator corresponding to a parameter-elliptic boundary value problem generates an analytic $C_0$-semigroup in $L^2(\Omega)$. We will apply these results to $\Delta^2$ and $\Delta$ in $\Omega_1$ and $\Omega_2$, respectively, with different boundary operators.

\begin{lemma}
  \label{4.1}
  The operator $-\Delta^2$ in $\Omega_1$, supplemented with the boundary operators $\BB_1$ and $\BB_2$ on $\partial \Omega_1$, is parameter-elliptic. The same holds for $-\Delta^2$ with clamped boundary conditions $u=\partial_\nu u=0$ on $\partial\Omega_1$ and for $-\Delta^2$ with boundary conditions $u=\BB_1 u =0$ on $\partial \Omega_1$.
\end{lemma}

\begin{proof}
  Obviously, the operator $-\Delta^2$ with symbol $-(\xi_1^2+\xi_2^2)^2$ is parameter-elliptic. Let $x_0\in\partial\Omega_1$, and choose a coordinate system associated with $x_0$. Then the $x_1$-axis is in tangential direction, while the positive $x_2$-axis coincides with the inner normal direction. In these coordinates, we have to solve the ordinary differential equation
  \begin{equation}
    \label{4-1}
    \begin{aligned}
     \md{\big(} \lambda + (\partial_2^2-\xi_1^2 )^2 \md{\big)} w(x_2) & = 0\quad (x_2\in (0,\infty)),\\
      \md{(}\BB_1(i\xi_1,\partial_2) w\md{)}(0) & = 0 ,\\
      \md{(}\BB_2(i\xi_1,\partial_2) w\md{)}(0) & = 0.
    \end{aligned}
  \end{equation}
  By the definition of the boundary operators $\BB_1$ and $\BB_2$, we obtain the local symbols $\md{\BB_1(i\xi_1,\partial_2) w = (\partial_2^2-\mu\xi_1^2)w}$ and $\md{\BB_2(i\xi_1,\partial_2) w = \big(-\partial_2^3 + (2-\mu)\xi_1^2 \partial_2\big)w}$. Now we use the following identity for $w\in H^2((0,\infty))$,
which is  obtained by integration by parts in $(0,\infty)$:
  \begin{equation}
    \label{4-2}
    \begin{aligned}
      \langle (\partial_2^2 & -\xi_1^2)^2w,w\rangle_{L^2((0,\infty))} = \mu \|(\partial_2^{\md{2}}-\xi_1^2)w\|_{L^2((0,\infty))}^2 \\
      & + (1-\mu)\Big( \|\xi_1^{\md{2}} w\|_{L^2((0,\infty))}^2 +\|\partial_2^2 w\|_{L^2((0,\infty))}^2
      + 2\|\xi_1\partial_2 w\|_{L^2((0,\infty))}^2 \Big)\\
      & + \md{\big(} \BB_1(i\xi_1,\partial_2)\md{w}\md{\big)}(0)\, \bar{\partial_2\md{w}(0)}
      + \md{\big(} \BB_2(i\xi_1,\partial_2)\md{w}\md{\big)}(0)\,\bar{\md{w}(0)}.
    \end{aligned}
  \end{equation}
  Note that this can be seen as a localized version of \eqref{Eq-Green-formula-bilaplace-mu}.

  Let $w$ be a stable solution of \eqref{4-1}. We multiply the first line in \eqref{4-1} by $\bar{w(x_2)}$ and integrate over $x_2\in (0,\infty)$. Due to the boundary conditions, all boundary terms in \eqref{4-2} disappear, and we obtain
\begin{align*}
  0 & = \langle (\lambda+  (\partial_2^2-\xi_1^2 )^2\md{)} w, w\rangle_{L^2((0,\infty))} \\
  & = \lambda \|w\|_{L^2((0,\infty))}^2 + \mu \|(\partial_2^{\md{2}}-\xi_1^2)w\|_{L^2((0,\infty))}^2 \\
  & + (1-\mu)\Big( \|\xi_1^{\md{2}} w\|_{L^2((0,\infty))}^2 +\|\partial_2^2 w\|_{L^2((0,\infty))}^2
      + 2\|\xi_1\partial_2 w\|_{L^2((0,\infty))}^2 \Big).
\end{align*}
As $\Re\lambda\ge 0$ and $\mu\in (0,1)$, we can take the real part and obtain $\|\xi_1^{\md{2}} w\|_{L^2((0,\infty))}=0$ and therefore $w=0$ in the case $\xi_1\not=0$. If $\xi_1=0$, then $\lambda\not=0$, and we obtain $\lambda\|w\|^{\md{2}}_{L^2((0,\infty))}=0$ which again implies $w=0$. Therefore, the Shapiro-Lopatinskii condition (SL) holds.

The statement for the other combinations of boundary conditions follows exactly in the same way, as in all cases the boundary terms in \eqref{4-2} disappear.
\end{proof}

We will apply parameter-elliptic theory to a boundary value problem in $\Omega_1$ with clamped boundary conditions on $\Gamma$ and free boundary conditions on $I$. In the next lemma, we show that the resolvent of such boundary value problems with `mixed' boundary conditions exists and satisfies a uniform estimate.

\begin{lemma}
  \label{4.2} Consider the boundary value problem
  \begin{equation}
    \label{4-3}
    \begin{aligned}
      (\lambda+\Delta^2) u & = f \quad \text{ in }\Omega_1,\\
      u = \partial_\nu u & =0 \quad \text{ on }\Gamma,\\
      \BB_1 u =  \BB_2 u & = 0 \quad\text{ on }I.
    \end{aligned}
  \end{equation}
  Then there exists a $\lambda_0>0$ such that for all $\lambda\ge\lambda_0$ and for all $f\in L^2(\Omega_1)$ there exists a unique solution $u\in H^4(\Omega_1)$ of \eqref{4-3}. Moreover, for all $\lambda\ge \lambda_0$ the uniform a priori-estimate
  \begin{equation}
    \label{4-4}
    \|u\|_{H^4(\Omega_1)} + \lambda\,\|u\|_{L^2(\Omega_1)} \le C_1 \|f\|_{L^2(\Omega_1)}
  \end{equation}
  holds with a constant $C_1$ depending on $\lambda_0$ but not on $\lambda$ or $f$.
\end{lemma}

\begin{proof}
  \textbf{(i)} We first show the existence of a solution. Let $f\in L^2(\Omega_1)$. We choose $\varphi_1\in C^\infty(\bar{\Omega_1})$ with $0\le \varphi_1\le 1$, $\varphi_1 =1$ in a neighbourhood of $\Gamma$, and $\supp\varphi_{\md{1}}\cap I = \emptyset$. We set $\varphi_2 := 1-\varphi_1$ on $\bar{\Omega_1}$. Further, let $\psi_j\in C^\infty(\bar\Omega_1)$, $j=1,2$, with $\md{0\le \psi_j\le 1}$, $\psi_j = 1$ on $\supp\varphi_j$, $\supp\psi_1 \cap I = \emptyset$, and $\supp\psi_2 \cap \Gamma = \emptyset$.

  By Lemma~\ref{4.1}, the boundary value problem given by $-\Delta^2$ and clamped boundary conditions is parameter-elliptic. Therefore (see \cite[Theorem~5.1]{Agranovich-Vishik64}) for $\lambda\ge \lambda_0$ with sufficiently large $\lambda_0$ there exists a unique solution $u^{(1)} = R_1(\lambda)\psi_1 f$ of
  \begin{align*}
    (\lambda+\Delta^2) u^{(1)} & = \psi_1 f\quad\text{ in } \Omega_1,\\
    u^{(1)} = \partial_\nu u^{(1)} & = 0 \quad\text{ on }\partial\Omega_1.
  \end{align*}
  In the same way, using parameter-ellipticity of the boundary value problem $(-\Delta^2, \BB_1,\BB_2)$, there exists a unique solution $u^{(2)}= R_2(\lambda)\psi_2 f$ of
  \begin{align*}
    (\lambda+\Delta^2) u^{(2)} & = \psi_2 f\quad\text{ in }\Omega_1,\\
    \BB_1 u^{(2)} = \BB_2 u ^{(2)} & = 0 \quad\text{ on }\partial\Omega_1.
  \end{align*}
  Moreover, the a priori-estimate
  \begin{equation}
    \label{4-5}
    \|u^{(j)}\|_{H^4(\Omega_1)} + \lambda\,\|u^{(j)}\|_{L^2(\Omega_1)} \le c_2 \|\psi_j f\|_{L^2(\Omega_1)}
  \end{equation}
  holds for all $\lambda\ge\lambda_0$ with a constant $c_2$ independent of $\lambda$ and $f$ (see \cite[Theorem~4.1]{Agranovich-Vishik64}).

  For $\lambda\ge\lambda_0$, we define
  \[ R(\lambda) f := \varphi_1 R_1(\lambda)\psi_1 f + \varphi_2 R_2(\lambda)\psi_2 f.\]
  By the product rule,
  \begin{align*}
    (\lambda+\Delta^2) R(\lambda) f & = \varphi_1 (\lambda+\Delta^2) R_1(\lambda)\psi_1f + \varphi_2 (\lambda+\Delta^2) R_2(\lambda)\psi_2 f \\
    & \quad + S_1(D) R_1(\lambda)\psi_1 f + S_2(D) R_2(\lambda)\psi_2 f,
  \end{align*}
  where $S_1(D)$ and $S_2(D)$ are linear partial differential operators of order $3$ depending on the choice of $\varphi_1$, but not on $\lambda$ or $f$. As $(\lambda+\Delta^2)R_j(\lambda)\psi_j f = \psi_j f$ and $\varphi_j\psi_j=\varphi_j$, $j=1,2$, we obtain
  \begin{equation}
    \label{4-6}
    (\lambda+\Delta^2)R(\lambda) f = (1+T(\lambda)) f
  \end{equation}
  with $T(\lambda)f := S_1(D) R_1(\lambda)\psi_1 f + S_2(D)R_2(\lambda)\psi_2 f$. As $S_j(D)$ are bounded linear operators from $H^3(\Omega_1)$ to $L^2(\Omega_1)$, we can estimate
  \begin{align*}
    \|S_j(D)R_j(\lambda)\psi_j f\|_{L^2(\Omega_1)} & \le C \|R_j(\lambda)\psi_j f\|_{H^3(\Omega_1)} \\
    & \le \delta \|R_j(\lambda)\psi_j f\|_{H^4(\Omega_1)} + C_\delta \|R_j(\lambda) \psi_j f\|_{L^2(\Omega_1)}\\
    & \le c_2 \delta\|f\|_{L^2(\Omega_1)} + \frac{c_2}{\lambda} \, C_\delta \|f\|_{L^2(\Omega_1)}.
  \end{align*}
  Here we used the interpolation inequality and \eqref{4-5}. Now we first choose $\delta>0$ small enough such that $c_2\delta\le \frac14$ and then $\lambda\ge\lambda_0$ with $\lambda_0$ large enough such that $\frac{\md{c_2}}\lambda\, C_\delta \le \frac14$. Therefore, the norm of $T(\lambda)$ as a bounded operator in $L^2(\Omega_1)$ is not larger than $\frac12$, and  $1+T(\lambda)$ is invertible. So we can define
  \[ u := R(\lambda) (1+T(\lambda))^{-1} f\in H^4(\Omega_1).\]
  From \eqref{4-6} we see $(\lambda+\Delta^2) u = f$, and by definition of $R(\lambda)$ we have
  \[ u|_{\Gamma} = (R_1(\lambda)\psi_1 f)|_\Gamma = 0,\; \partial_\nu u|_{\Gamma} =0 \]
    as well as
  \[ \BB_j u|_I = \BB_j \big( \varphi_2 R_2(\lambda)\psi_2 f\big)|_I = \BB_j R_2(\lambda)\psi_2 f|_I = 0\; (j=1,2).\]
  Therefore, $u$ is a solution of the boundary value problem \eqref{4-3}.

  \textbf{(ii)} Now we show that every solution of \eqref{4-3} satisfies the a priori-estimate  \eqref{4-4}. Let $u\in H^4(\Omega_1)$ be a solution of \eqref{4-3}. Then $u^{(1)} := u\md{\varphi_1} $ is a solution of the boundary value problem
  \begin{align*}
    (\lambda+\Delta^2) u^{(1)} & = \md{\varphi_1} f + \tilde S_1(D) u \quad\text{ in }\Omega_1,\\
    u^{(1)} = \partial_\nu u^{(1)} & = 0 \quad\text{ on }\partial\Omega_1,
  \end{align*}
  where $\tilde S_1(D)$ is a linear partial differential operator of order $3$.
  By parameter-elliptic theory \cite[Theorem~4.1]{Agranovich-Vishik64}, $u^{(1)}$ satisfies
  \[ \|u^{(1)}\|_{H^4(\Omega_1)} + \lambda\,\| u^{(1)}\|_{L^2(\Omega_1)} \le C\big( \|f\|_{L^2(\Omega_1)} + \|u\|_{H^3(\Omega_1)}\big).\]
  The same holds for $u^{(2)} := \md{\varphi_2} u$ by parameter-ellipticity of $(-\Delta^2, \BB_1,\BB_2)$. For the sum $u=u^{(1)}+u^{(2)}$, we get
  \[ \|u\|_{H^4(\Omega_1)} + \lambda \|u\|_{L^2(\Omega_1)} \le C \big( \|f\|_{L^2(\Omega_1)} + \|u\|_{H^3(\Omega_1)}\big).\]
  Now, by interpolation inequality again, we can estimate
  \[ \|u\|_{H^4(\Omega_1)} + \lambda \|u\|_{L^2(\Omega_1)} \le C  \|f\|_{L^2(\Omega_1)} + \delta \|u\|_{H^4(\Omega_1)}+ C_\delta \|u\|_{L^2(\Omega_1)}.\]
  Choosing $\delta\le\frac12$ and then $\lambda_0>2C_\delta$, we can absorb the $u$-dependent terms on the right-hand side and obtain
  \[ \|u\|_{H^4(\Omega_1)} + \lambda\,\|u\|_{L^2(\Omega_1)} \le C \|f\|_{L^2(\Omega_1)}\quad (\lambda\ge\lambda_0).\]
  This also yields uniqueness of the solution.
\end{proof}

\begin{corollary}
  \label{4.3}
  Let $f\in L^2(\Omega_1)$, $g_1\in H^{7/2}(\Gamma)$, $g_2\in H^{5/2}(\Gamma)$, $h_1\in H^{3/2}(I)$, and $h_2\in H^{1/2}(I)$. Then for sufficiently large $\lambda_0>0$, the boundary value problem
  \begin{equation}
    \label{4-7}
    \begin{aligned}
      (\lambda_0+\Delta^2) u & = f\quad\text{ in }\Omega_1,\\
      u & = g_1 \quad\text{ on }\Gamma,\\
      \partial_\nu u & = g_2\quad\text{ on }\Gamma,\\
      \BB_1 u & = h_1 \quad\text{ on }I,\\
      \BB_2 u & = h_2 \quad\text{ on }I
    \end{aligned}
  \end{equation}
  has a unique solution $u\in H^4(\Omega_1)$. Moreover, the a priori-estimate
 \begin{equation}
    \label{4-8}
    \begin{aligned}
    \|u\|_{H^4(\Omega_1)} & \le C_2 \Big(  \|f\|_{L^2(\Omega_1)} + \|g_1\|_{H^{7/2}(\Gamma)} + \|g_2\|_{H^{5/2}(\Gamma)} \\
    & \quad + \|h_1\|_{H^{3/2}(I)} + \|h_2\|_{H^{1/2}(I)}\Big)
    \end{aligned}
  \end{equation}
  holds with a constant $C_2>0$ which depends on $\lambda_0$ but not on $u$ or on the data.
\end{corollary}

\begin{proof}
  We define $G:= (g_1,g_2,0,0)^\top$ on $\Gamma$ and $H:=(0,0,h_1,h_2-(\div \nu)h_1)^\top$ on $I$. By \cite{Triebel78}, Section~4.7.1, p.~330, the map
  \[\mathcal R\colon u\mapsto \big( u|_{\partial\Omega_1}, \partial_\nu u|_{\partial\Omega_1},  \partial_\nu^2 u|_{\partial\Omega_1}, \partial_\nu^3 u|_{\partial\Omega_1},\big)^\top\]
  is a retraction from $H^4(\Omega_1)$ to $\prod_{j=0}^3 H^{4-j-1/2}(\partial\Omega_1)$. Let $\mathcal E$ denote  a coretraction to $\mathcal R$, and set
  \[ u^{(1)} := \mathcal E \big(\chi_\Gamma G +\chi_I H\big) \in H^4(\Omega_1).\]
  The boundary operators $\BB_1$ and $\BB_2$ can be expressed in terms of normal and tangential derivatives as (see \cite{MR1745475}, Propositions 3C.7 and 3C.11)
  \begin{align*}
   \BB_1 u^{(1)} & = \partial_\nu^2  u^{(1)} + \mu \partial_\tau^2 u^{(1)} + \mu (\mathrm{div}\,\nu)\partial_\nu u^{(1)}  ,\\
   \BB_2 u^{(1)} & = \partial_\nu^3 u^{(1)} + \partial_\nu \partial_\tau^2 u^{(1)} + (1-\mu)\partial_\tau \partial_\nu \partial_\tau u^{(1)} + \partial_\nu[(\div \nu)\partial_\nu u^{(1)}].
   \end{align*}
   As $u^{(1)}=\partial_\nu u^{(1)}=0$ on $I$ due to the definition of $u^{(1)}$, we obtain $\partial_\tau^k u^{(1)} = \partial_\tau^k \partial_\nu u^{(1)} = 0$ on $I$ for all $k\in \N$. Moreover, applying  the identity
   \[ \partial_\nu\partial_\tau w = \partial_\tau\partial_\nu w - (\div\nu) \partial_\tau w \]
   (see \cite[Corollary~3C.10]{MR1745475}) to $w:= u^{(1)}$ and to $w:= \partial_\tau u^{(1)}$, respectively, we see that
   \[ \partial_\nu\partial_\tau^2 u^{(1)} = \partial_\tau\partial_\nu\partial_\tau u ^{(1)} =0 \quad\text{ on }I.\]
   Therefore,
   \begin{align*}
     \BB_1 u^{(1)} & = \partial_\nu^2  u^{(1)} = h_1,\\
     \BB_2 u^{(1)} & = \partial_\nu^3 u^{(1)} + (\div \nu) \partial_\nu^2 u^{(1)} = h_2
   \end{align*}
   on $I$.    By continuity of $\mathcal E$, we have
   \begin{equation}
     \label{4-9}
     \begin{aligned}
       \|(\lambda_0 & +\Delta^2) u^{(1)}\|_{L^2(\Omega_1)} \le C \|u^{(1)}\|_{H^4(\Omega_1)} \\
       & \le C \Big( \|g_1\|_{H^{7/2}(\Gamma)} + \|g_2\|_{H^{5/2}(\Gamma)}  + \|h_1\|_{H^{3/2}(I)} + \|h_2\|_{H^{1/2}(I)}\Big)
     \end{aligned}
   \end{equation}
   with $C$ depending only on $\lambda_0$.

   Considering $u^{(2)} := u- u^{(1)}$, we see that $u$ solves \eqref{4-7} if and only if $u^{(2)}$ solves the boundary value problem
\begin{align*}
  (\lambda_0+\Delta^2) u^{(2)} &= \tilde f\quad\text{ in }\Omega_1,\\
  u^{(2)} = \partial_\nu u^{(2)} & = 0 \quad\text{ on }\Gamma,\\
  \BB_1 u^{(2)} = \BB_2 u^{(2)} & = 0 \quad\text{ on }I.
\end{align*}
Here, $\tilde f := f - (\lambda_0+\Delta^2) u^{(1)}$. By Lemma~\ref{4.2}, this is uniquely solvable, and the a priori estimate $\|u^{(2)}\|_{H^4(\Omega_1)} \le C \|\tilde f\|_{L^2(\Omega_1)}$ in connection with \eqref{4-9} yields \eqref{4-8}.
\end{proof}

\begin{remark}
  \label{4.4}
  a) The statement and the proof of Lemma~\ref{4.2} and Corollary~\ref{4.3} are independent of the particular equation. We have shown unique solvability and uniform a priori-estimates for boundary value problems where we have different boundary operators on disjoint and not connected parts of the boundary, given that on each part of the boundary the Shapiro-Lopatinskii condition holds.

  b) From elliptic theory, it is well known that the analog statement of Corollary~\ref{4.3} also holds (with $\lambda_0=0$) in the much easier situation of the Dirichlet Laplacian in $\Omega_2$: For every $f\in L^2(\Omega_2)$ and $g\in H^{3/2}(I)$ there exists a unique $u\in H^2(\Omega_2)$ with $\Delta u = f$ in $\Omega_2$ and $u|_I = g$, and $\|u\|_{H^2(\Omega_2)} \le C(\|f\|_{L^2(\Omega_2)} + \|g\|_{H^{3/2}(I)})$.
\end{remark}

The elliptic regularity results above are the key for the strong solvability of the transmission problem, i.e. for higher regularity of the weak solution.

\begin{theorem}
  \label{4.5} Let $U=(u_1,v_1,u_2,v_2)^\top \in D(\AA)$. Then $u_1\in H^4(\Omega_1)$ and $u_2\in H^2(\Omega_2)$. In particular, the transmission conditions hold in the strong sense of traces on the interface $I$.
\end{theorem}

\begin{proof}
  Let $U\in D(\AA)$ and $F = (f_1,g_1,f_2,g_2)^\top := \AA U$. Then $v_1=f_1\in H^2_\Gamma(\Omega_1)$, $v_2=f_2\in H^1(\Omega_2)$, $\Delta u_2 = g_2+\beta f_2$, and $\Delta^2 u_1 = \rho\Delta f_1 - g_1$.

  By Remark~\ref{4.4} b), there exists a unique $\tilde u_2\in H^2(\Omega_{\md{2}})$ such that $\Delta \tilde u_2 = g_2+\beta f_2$ in $\Omega_1$ and $\tilde u_2|_I = u_1|_I$. As $u_2-\tilde u_2$ belongs to $H^1_0(\Omega_{\md{2}})$ and is a weak solution of $\Delta(u_2-\tilde u_2)=0$, we immediately obtain $\tilde u_2 = u_2$ which already yields $u_2\in H^2(\Omega_2)$.

  Similarly, by Corollary~\ref{4.3} there exists a unique solution $\tilde u_1\in H^4(\Omega_1)$ of the boundary value problem
  \begin{equation}
    \label{4-10}
    \begin{aligned}
      (\lambda_0 + \Delta^2) \tilde u_1 & = \lambda_0 u_1 + \rho\Delta f_1 - g_1\quad \text{ in }\Omega_1,\\
      \tilde u_1 = \partial_\nu \tilde u_1 & = 0 \quad\text{ on }\Gamma,\\
      \BB_1 \tilde u_1 = 0,\; \BB_2 \tilde u_1 & = \rho\partial_\nu f_1 \md{-} \partial_\nu u_2\quad\text{ on } I.
    \end{aligned}
  \end{equation}
  Note here that $\lambda_0 u_1+\rho\Delta f_1 - g_1 \in L^2(\Omega_1)$ and $\rho\partial_\nu f_1 + \partial_\nu u_2\in H^{1/2}(I)$, and that all boundary conditions hold in the trace sense.

  Let $\psi_1\in H^2_\Gamma(\Omega_1)$. Then \eqref{Eq-Green-formula-bilaplace-mu} in combination with the boundary conditions above yields
  \begin{equation}
    \label{4-11}
    \begin{aligned}
      \langle (\lambda_0 + \Delta^2) \tilde u_1,\psi_1\rangle_{L^2(\Omega_1)} & = \lambda_0 \langle \tilde u_1, \psi_1\rangle_{L^2(\Omega_1)} + \langle \tilde u_1,\psi_1\rangle_{H^2_\Gamma(\Omega_1)}\\
      & \quad + \langle \rho\partial_\nu f_1 \md{-} \partial_\nu u_2,\psi_1\rangle_{L^2(I)}.
    \end{aligned}
  \end{equation}

  We compare $\tilde u_{\md{1}}$ with the weak solution $u_{\md{1}}$. For this, we consider $\Phi := (0,\psi_1,0,\psi_2)^\top$ with $\psi_1\in H^2_\Gamma(\Omega_1)$, $\psi_2\in H^1(\Omega_{\md{2}})$, and $\psi_1=\psi_2$ on $I$. By definition of $D(\AA)$, we obtain
  \begin{align*}
    \langle \AA U,\Phi\rangle_{\HH} &  = \langle -\Delta^2 u_1 +\rho\Delta v_1, \psi_1\rangle_{L^2(\Omega_1)} + \langle \Delta u_2-\beta v_2,\psi_2\rangle_{L^2(\Omega_2)}\\
    & = - \langle u_1,\psi_1\rangle_{H^2_\Gamma(\Omega_1)} - \rho\langle\nabla v_1,\nabla \psi_1\rangle_{L^2(\Omega_1)^2} \\
    & \quad - \langle \nabla u_2,\nabla\psi_2\rangle_{L^2(\Omega_2)^2} - \beta \langle v_2,\psi_2\rangle_{L^2(\Omega_2)}.
  \end{align*}
  From this,  $v_1=f_1$ and integration by parts (as we already know $u_2\in H^2(\Omega_2)$), we see that
  \begin{align*}
    \langle \Delta^2 u_1,\psi_1\rangle_{L^2(\Omega_1)} & = \langle\rho\Delta v_1,\psi_1\rangle_{L^2(\Omega_1)} + \langle \Delta u_2,\psi_2\rangle_{L^2(\Omega_2)} \\
    & \quad + \langle u_1,\psi_1\rangle_{H^2_\Gamma(\Omega_1)} + \langle\rho\nabla v_1,\nabla \psi_1\rangle_{L^2(\Omega_1)^2} + \langle\nabla u_2,\nabla \psi_2\rangle_{L^2(\Omega_2)^2} \\
    & = \langle \rho \partial_\nu f_1,\psi_1\rangle_{L^2(I)} \md{-} \langle \partial_\nu u_2,\psi_2\rangle_{L^2(I)} \md{+ \langle u_1,\psi_1\rangle_{H^2_\Gamma(\Omega_1)}} \\
    & = \langle \rho\partial_\nu f_1 \md{-} \partial_\nu u_2, \psi_1\rangle_{L^2(I)} \md{+ \langle u_1,\psi_1\rangle_{H^2_\Gamma(\Omega_1)}}.
  \end{align*}
  In the last step we used $\psi_1=\psi_2$ on $I$. Therefore, \eqref{4-11} also holds with $\tilde u_1$ being replaced by $u_1$.

  By definition of $\tilde u_1$, we have $(\lambda_0+\Delta^2)\tilde u_1 = (\lambda_0+\Delta^2) u_1 = \lambda_0 u_1 + \rho\Delta f_1 - g_1$. Therefore, we can insert the difference $w:= \tilde u_1 - u_1$ into \eqref{4-11} and obtain $0 = \md{\lambda_0\langle w, \psi_1\rangle_{L^2(\Om_1)} + } \langle w,\psi_1\rangle_{H^2_\Gamma(\Omega_1)}$ for all $\psi_1\in H^2_\Gamma(\Omega_1)$. But by construction $w\in H^2_\Gamma(\Omega_1)$, so we  can set $\psi_1:= w$ and get $w=0$, i.e., $u_1 = \tilde u_1\in H^4(\Omega_1)$.
\end{proof}

\section{Polynomial stability}

As we saw in Section \ref{expstab}, the system is not exponentially stable when $\beta=0$. When $\beta=\rho=0$, \eqref{dissieq} shows that the system is conservative. In this section we consider the case $\beta=0$ and $\rho>0$ and show that polynomial decay is still guaranteed under certain geometrical conditions. More precisely, we assume that there exists some $x_0\in \R^2$ such that
\begin{align}\label{eq:geometric}
	q \cdot \nu = q^\top \nu \leq 0
\end{align}
on $I$, where $q(x):=x-x_0$. Note that $\nu$ is the inner normal w.r.t. to $\Om_2$, which is why we require $q\cdot \nu\leq 0$ instead of $q \cdot \nu\geq 0$. In order to prove the polynomial stability, we use the following result by Borichov and Tomilov (Theorem 2.4 in \cite{Borichev-Tomilov10})

\begin{theorem}\label{lemma:pol1}
	Let $(T(t))_{t\geq 0}$ be a bounded $C_0$-semigroup on a Hilbert space $H$ with generator $A$ such that $i\R \subset \rho(A)$. Then, for fixed $\alpha > 0$ the following
	conditions are equivalent:
	\begin{enumerate}
	[$(i)$]
		\item There exist $C > 0$ and $\lambda_0 > 0$ such that for all $\lambda \in \R$ with $\vert \lambda \vert > \lambda_0$ and all $F \in H$ it holds
			\begin{align*}
				\Vert (i\lambda - A)^{-1}F \Vert \leq C\vert \lambda \vert^\alpha \Vert F \Vert.
			\end{align*}
		\item There exists some $C > 0$ such that for all $t > 0$ and all $U_0 \in D(A)$ it holds
			\begin{align*}
				\Vert T(t)U_0 \Vert \leq Ct^{-\frac{1}{\alpha}} \Vert AU_0 \Vert.
			\end{align*}
	\end{enumerate}
\end{theorem}

We now state the main result of this section: we show polynomial stability for the transmission problem in the case where only the plate equation is damped but the wave equation
is undamped. Using rather general methods, it is very likely that the rate of decay is not optimal. On the other hand, the approach might be versatile enough to be applicable to different transmission problems of a similar form, i.e. transmission problems where the equation in the outer domain is parameter-elliptic, whereas the equation in the inner domain simply is of lower order.

\begin{theorem}\label{thm:pol}
	Let $\beta = 0$ and $\rho > 0$ and assume that the geometrical condtion \eqref{eq:geometric} is satisfied. Then the semigroup $(S(t))_{t\geq0}$ decays polynomially
	of order at least $1/30$, i.e. there exists some constant $C>0$ such that
 	\[
  		\|S(t)\|_\HH\leq Ct^{-\frac{1}{30}}\|\AA U_0\|_\HH
 	\]
	for all $t>0$ and $U_0\in D(\AA)$.
\end{theorem}

Throughout the remainder of this section, let $\lambda_0 > 0$, $\lambda \in \R$ with $\vert \lambda \vert > \lambda_0$, $F=(f_1,g_1,f_2,g_2)^\top \in \HH$ and
$U=(u_1,v_1,u_2,v_2)^\top \in D(\AA)$ such that $(i\lambda-\AA)U=F$.
We first observe that $(i\lambda-\AA)U=F$ implies
\begin{eqnarray}
	\label{pol1} v_1&=&i\lambda u_1-f_1,\\
	\label{pol2} -\lambda^2 u_1+\Delta^2 u_1-i\lambda\rho \Delta u_1&=&g_1+i\lambda f_1-\rho \Delta f_1,\\
	\label{pol3} v_2&=&i\lambda u_2-f_2,\\
	\label{pol4} -\lambda^2u_2-\Delta u_2&=&g_2+i\lambda f_2.
\end{eqnarray}
Multiplying \eqref{pol2} by $-\overline{u_1}$ and \eqref{pol4} by $-\overline{u_2}$, integrating and adding yields
\begin{align*}
 	\lambda^2&\big(\|u_1\|_{L^2(\Om_1)}^2+\|u_2\|_{L^2(\Om_2)}^2\big) -\langle\Delta^2 u_1,u_1\rangle_{L^2(\Om_1)}\\
 	&\quad +i\lambda\rho \langle\Delta u_1,u_1\rangle_{L^2(\Om_1)}+\langle \Delta u_2,u_2\rangle_{L^2(\Om_2)}\\
 	&=-\langle g_1+i\lambda f_1-\rho \Delta f_1,u_1\rangle_{L^2(\Om_1)}-\langle g_2+i\lambda f_2,u_2\rangle_{L^2(\Om_2)}.
\end{align*}
Using Lemma \ref{Lemma-Green-formula-bilaplace-mu}, integration by parts and plugging in the boundary and transmission conditions we obtain
\begin{align*}
 	\lambda^2&\big(\|u_1\|_{L^2(\Om_1)}^2+\|u_2\|_{L^2(\Om_2)}^2\big)-\|u_1\|_{H^2_\Gamma(\Om_1)}^2-i\lambda\rho\|\nabla u_1\|_{L^2(\Om_1)^2}^2-\|\nabla u_2\|_{L^2(\Om_2)^2}^2\\
 	& =-\rho\langle \nabla f_1,\nabla u_1\rangle_{L^2(\Om_1)^2}-\langle g_1+i\lambda f_1,u_1\rangle_{L^2(\Om_1)}-\langle g_2+i\lambda f_2,u_2\rangle_{L^2(\Om_2)}.
\end{align*}
Taking the real part in the above equality we see that
\begin{align*}
	\|u_1\|_{H^2_\Gamma(\Om_1)}^2+\|\nabla u_2\|_{L^2(\Om_2)^2}^2&\leq \lambda^2\big(\|u_1\|_{L^2(\Om_1)}^2+\|u_2\|_{L^2(\Om_2)}^2\big)\\
 	&\quad +\big(|\lambda|\|f_1\|_{L^2(\Om_1)}+\|g_1\|_{L^2(\Om_1)}\big)\|u_1\|_{L^2(\Om_1)}\\
 	&\quad +\big(|\lambda|\|f_2\|_{L^2(\Om_2)}+\|g_2\|_{L^2(\Om_1)}\big)\|u_2\|_{L^2(\Om_2)}\\
 	&\quad +\rho \|\nabla f_1\|_{L^2(\Om_1)^2}\|\nabla u_1\|_{L^2(\Om_1)^2}\\
 	&\leq \lambda^2\big(\|u_1\|_{L^2(\Om_1)}^2+\|u_2\|_{L^2(\Om_2)}^2\big) + C |\lambda| \|U\|_{\HH}\|F\|_{\HH},
\end{align*}
where we used the fact that $|\la|\geq \la_0$. Moreover, due to \eqref{pol1} and \eqref{pol3}, we have that
\[
	\|v_j\|_{L^2(\Om_j)}^2=\|i\la u_j-f_j\|_{L^2(\Om_j)}^2\leq 2\big(\la^2\|u_j\|_{L^2(\Om_j)}^2+\|f_j\|_{L^2(\Om_j)}^2\big)
\]
for $j=1,2$. Hence,
\begin{equation}\label{pol5}
 	\|v_1\|_{L^2(\Om_1)}^2+\|v_2\|_{L^2(\Om_2)}^2
 		\leq C \left( \la^2\big(\|u_1\|_{L^2(\Om_1)}^2+\|u_2\|_{L^2(\Om_2)}^2\big)+ \|F\|_{\HH}^2 \right)
\end{equation}
and therefore, combining \eqref{pol5} with the estimate for $u_1$ and $u_2$, we get that
\[
 	\|U\|_{\HH}^2\leq C \left(\lambda^2\big(\|u_1\|_{L^2(\Om_1)}^2+\|u_2\|_{L^2(\Om_2)}^2\big) + \big(|\lambda| \|U\|_{\HH}\|F\|_{\HH}+\|F\|_{\HH}^2\big) \right).
\]
It remains to estimate $\|u_1\|_{L^2(\Om_1)}^2$ and $\|u_2\|_{L^2(\Om_2)}^2$. In order to estimate $\|u_1\|_{L^2(\Om_1)}^2$, we observe that, due to \eqref{dissieq}, it holds
\[
 	\|\nabla v_1\|_{L^2(\Om_1)^2}^2\leq \frac{1}{\rho}\|U\|_{\HH}\|F\|_{\HH},
\]
and therefore, using Poincar\'e's inequality, we obtain that
\begin{align}
 	\la^2\|u_1\|_{H^1(\Om_1)}^2 = \|f_1+v_1\|_{H^1(\Om_1)}^2
 		& \leq \md{2}\big(\|f_1\|_{H^1(\Om_1)}^2+\|v_1\|_{H^1(\Om_1)}^2\big)		\notag\\
 		& \leq C\big(\|F\|_{\HH}^2+\|\nabla v_1\|_{L^2(\Om_1)^2}^2\big) 			\notag \\
 		& \leq C\big(\|U\|_{\HH}\|F\|_{\HH}+\|F\|_{\HH}^2\big).						\label{pol8}
 \end{align}
Using the fact that $|\la|\geq \la_0$ and $\|u_1\|_{L^2(\Om_1)}\leq \|u_1\|_{H^1(\Om_1)}$, we thus get that
\begin{equation}\label{pol7}
 	\|U\|_{\HH}^2\leq C \left(\lambda^2\|u_2\|_{L^2(\Om_2)}^2 + \big(|\lambda| \|U\|_{\HH}\|F\|_{\HH}+\|F\|_{\HH}^2\big) \right)
\end{equation}
and it remains to estimate $\|u_2\|_{L^2(\Om_2)}^2$, which will be done in the following lemmas.

\begin{lemma}\label{lemma:pol2}
	It holds
	\begin{align*}
		\la^2\|u_2\|_{L^2(\Om_2)}^2\leq C\big(|\la|\|U\|_{\HH}\|F\|_\HH +\|F\|_{\HH}^2\big) +\int_I\big|\partial_\nu u_2(q\nabla \overline{u_2})\big|\, {\rm d}S.
	\end{align*}
\end{lemma}

\begin{proof}
 	Using Rellich's identity (cf. \cite{Mitidieri93}, Eq. (2.5)), we have that
	\begin{equation}\label{pol6}
 		\Re \int_{\Om_2}\Delta u_2 (q\nabla \overline{u_2})\, {\rm d}x=-\Re \int_I\partial_\nu u_2(q\nabla \overline{u_2})-\frac{1}{2}(q\cdot\nu)|\nabla u_2|^2\, {\rm d}S.
	\end{equation}
	We multiply \eqref{pol4} by $q\nabla\overline{u_2}$, integrate over $\Om_2$, take the real part and use \eqref{pol6} in order to obtain that
	\begin{align*}
	 	\Re&\bigg(-\la^2\int_{\Om_2}u_2 (q\nabla \overline{u_2})\, {\rm d}x\bigg)
	 		+ \Re\bigg(\int_I \partial_\nu u_2(q\nabla \overline{u_2})-\frac{1}{2}(q\cdot \nu)|\nabla u_2|^2\, {\rm d}S\bigg)\\
	 		& = \Re\int_{\Om_2}(g_2 + i\lambda f_2)(q\nabla \overline{u_2})\, {\rm d}x.
	\end{align*}
	As $q\nabla u_2 = \operatorname{div}(qu_2) - 2u_2$, integration by parts and taking the real part shows
	\[
		\Re\int_{\Om_2}u_2(q\nabla \overline{u_2})\, {\rm d}x = -\|u_2\|_{L^2(\Om_2)}^2-\frac{1}{2}\int_I(q\nu)|u_2|^2\, {\rm d}S
	\]
	and we obtain
	\begin{align*}
	 	\la^2 \|u_2\|_{L^2(\Om_2)}^2&=-\frac{\la^2}{2}\int_I(q\cdot\nu)|u_2|^2\, {\rm d}S + \frac{1}{2}\int_I(q\cdot\nu)|\nabla u_2|^2\, {\rm d}S\\
	 	&\quad -\Re\int_I\partial_\nu u_2(q\nabla \overline{u_2})\, {\rm d}S + \Re\int_{\Om_2}(g_2 + i\la f_2)(q\nabla \overline{u_2})\, {\rm d}x.
	\end{align*}
	Since $q\cdot\nu\leq 0$ on $I$ and $u_1=u_2$ on $I$, we arrive at
	\begin{align*}
		\la^2 \|u_2\|_{L^2(\Om_2)}^2&\leq C\big(\la^2\|u_1\|_{H^1(\Om_1)}^2+|\la|\|U\|_{\HH}\|F\|_\HH\big)+\int_I\big|\partial_\nu u_2(q\nabla \overline{u_2})\big|\, {\rm d}S\\
	  	&\leq C\big(|\la|\|U\|_{\HH}\|F\|_\HH+\|F\|_{\HH}^2\big)+\int_I\big|\partial_\nu u_2(q\nabla \overline{u_2})\big|\, {\rm d}S,
	\end{align*}
	where in the first step we used the trace theorem and in the last step we used \eqref{pol8} as well as $|\la|\geq \la_0$.
\end{proof}

\begin{lemma}\label{lemma:pol3}
	For any $\varepsilon > 0$, there exists a constant $C(\varepsilon) > 0$ such that
 	\[
 		\int_I\big|\partial_\nu u_2(q\nabla \overline{u_2})\big|\, {\rm d}S\leq \varepsilon \|U\|_\HH^2+C(\varepsilon)|\la|^{60}\|F\|_\HH^2.
 	\]
\end{lemma}

\begin{proof}
	Using the transmission conditions, we can estimate
	\begin{align*}
		\int_I\big|\partial_\nu u_2(q\nabla \overline{u_2})\big|\, {\rm d}&S=\int_I\big|\BB_2 u_1-i\la \rho\partial_\nu u_1
	  		+\rho\partial_\nu f_1\big|\big|q\nabla \overline{u_2}\big|\, {\rm d}S\\\
	  	& \leq C \|\BB_2 u_1-i\la\rho \partial_\nu u_1+\rho\partial_\nu f_1\|_{L^2(I)}\|\nabla u_2\|_{L^2(I)^2}\\
	  	& \leq C\big(\|u_1\|_{H^{7/2}(\Om_1)}+|\la|\|u_1\|_{H^{3/2}(\Om_1)}+\|F\|_\HH\big)\|u_2\|_{H^{3/2}(\Om_2)}.
	\end{align*}
	In order to estimate the terms on the right-hand side, we will use interpolation theory
	for both the terms $\Vert u_1 \Vert_{H^{7/2}(\Om_1)}$ and $\Vert u_2 \Vert_{H^{3/2}(\Om_2)}$.
	Hence, we start with an estimate for $\Vert u_2 \Vert_{H^2(\Omega_2)}.$ \\
	By \eqref{pol4}, $u_2$ satisfies the equation
	\[
		\Delta u_2 = -(\lambda^2 u_2 + g_2 + i\lambda f_2).
	\]
	Therefore, Remark \ref{4.4} b) and $u_1=u_2$ on $I$ yield the estimate
	\begin{align}
	 	\Vert u_2 \Vert_{H^2(\Om_2)}
	 		& \leq C \big( \Vert \lambda^2 u_2 + g_2 + i\lambda f_2 \Vert_{L^2(\Om_2)} + \Vert u_1 \Vert_{H^{3/2}(I)} \big) 								\notag \\
	 		& \leq C \big( \lambda^2 \Vert u_2 \Vert_{L^2(\Om_2)} + \vert \lambda \vert \Vert F \Vert_\HH + \Vert u_1 \Vert_{H^2_\Gamma(\Om_1)} \big) 		\notag \\
	 		& \leq C \vert \lambda \vert \big( \vert \lambda \vert \Vert U \Vert_\HH + \Vert F \Vert_\HH \big).	\label{eq:u2H2}
	\end{align}
	Using interpolation inequality and the equivalence of the $p$-norms on $\R^2$, we get that
	\begin{align}
	 	\Vert u_2 \Vert_{H^{3/2}(\Om_2)}
	 		\leq C \Vert u_2 \Vert_{H^2(\Om_2)}^{1/2} \Vert u_2 \Vert_{H^1(\Om_2)}^{1/2}
	 		& \leq C \Vert u_2 \Vert_{H^2(\Om_2)}^{1/2} \Vert U \Vert_\HH^{1/2} 				\notag \\
	 		& \leq C \big( \vert \lambda \vert \Vert U \Vert_\HH + \vert \lambda \vert^{1/2} \Vert U \Vert_\HH^{1/2} \Vert F \Vert_\HH^{1/2} \big).	
	 	\label{eq:u2H32}
	\end{align}
	In the next step, we will estimate the term $\Vert u_1 \Vert_{H^{7/2}(\Om_1)}$.
	By \eqref{pol2}, $u_1$ satisfies the equation
	\[
		(\lambda + \Delta^2)u_1 = \lambda u_1 + \lambda^2 u_1 + i\lambda \rho \Delta u_1 + g_1 + i\lambda f_1 - \rho \Delta f_1.
	\]
	Hence, Corollary \ref{4.3} states
	\begin{align*}
		\Vert u_1 \Vert_{H^4(\Om_1)}
			& \leq C \big( \Vert \lambda u_1 + \lambda^2 u_1 + i\lambda \rho \Delta u_1 + g_1 + i\lambda f_1 - \rho \Delta f_1 \Vert_{L^2(\Om_1)} \\
				& \qquad \quad + \Vert \mathscr{B}_1 u_1 \Vert_{H^{7/2}(I)} + \Vert \mathscr{B}_2 u_1 \Vert_{H^{1/2}(I)} \big)
	\end{align*}
	due to the homogeneous boundary conditions on $\Gamma.$
	Using the trace theorem, the transmission conditions, \eqref{pol1} and \eqref{pol2} as well as \eqref{eq:u2H2}, we obtain
	\begin{align*}
		\Vert u_1 \Vert_{H^4(\Om_1)}
			& \leq C \left( \vert \lambda \vert \big(\vert \lambda \vert \Vert U \Vert_\HH + \Vert F \Vert_\HH \big) + \Vert \partial_\nu u_2 \Vert_{H^{1/2}(I)} \right) \\
			& \leq C \vert \lambda \vert \big( \vert \lambda \vert \Vert U \Vert_\HH + \Vert F \Vert_\HH \big).
	\end{align*}
	Moreover, note that \eqref{pol8} reformulates to
	\[
	 	\|u_1\|_{H^1(\Om_1)}\leq \frac{C}{|\la|}\big(\|U\|_\HH \|F\|_{\HH}+\|F\|_\HH^2\big)^{1/2}.
	\]
	Again, by interpolation inequality and the equivalence of the $p$-norms on $\R^2$, we thus get that
	\begin{align}
		\Vert u_1 \Vert_{H^{7/2}(\Om_1)}
			& \leq C \Vert u_1 \Vert_{H^4(\Om_1)}^{5/6} \Vert u_1 \Vert_{H^1(\Om_1)}^{1/6} \notag \\
			& \leq C \vert \lambda \vert^{5/6} \big( \vert \lambda \vert \Vert U \Vert_\HH
				+ \Vert F \Vert_\HH \big)^{5/6} \vert \lambda \vert^{-1/6} \big( \Vert U \Vert_\HH \Vert F \Vert_\HH  + \Vert F \Vert_\HH^2)^{1/12} \notag \\
			& \leq C \vert \lambda \vert ^{2/3} \big( \vert \lambda \vert^{5/6} \Vert U \Vert_\HH^{5/6} + \Vert F \Vert_\HH^{5/6} \big)
				\big( \Vert U \Vert_\HH^{1/12} \Vert F \Vert_\HH^{1/12} + \Vert F \Vert_\HH^{1/6} \big) \notag \\
			& \leq C \bigg( \vert \lambda \vert^{3/2} \big( \Vert U \Vert_\HH^{11/12} \Vert F \Vert_\HH^{1/12} + \Vert U \Vert_\HH^{5/6} \Vert F \Vert_\HH^{1/6} \big)\notag \\
				& \qquad \quad + \vert \lambda \vert^{2/3} \big( \Vert U \Vert_\HH^{1/12} \Vert F \Vert_\HH^{11/12} + \Vert F \Vert_\HH \big) \bigg).
					\label{eq:u1H72}
	\end{align}
	Young's inequality
	\begin{align*}
		a^{2-\alpha} b^\alpha \leq \varepsilon a^2 + C(\varepsilon) b^2
	\end{align*}
	for fixed $\alpha \in (0, 2)$ and $\varepsilon > 0$ arbitrary, yields
	\begin{align*}
		\vert \lambda \vert^{5/2} \Vert U \Vert^{23/12}_\HH \Vert F \Vert_\HH^{1/12}
			= \Vert U \Vert_\HH^{23/12} \big( \vert \lambda \vert^{30} \Vert F \Vert_\HH \big)^{1/12}
			\leq \varepsilon \Vert U \Vert_\HH^2 + C(\varepsilon) \vert \lambda \vert^{60} \Vert F \Vert_\HH^2.
	\end{align*}
	Considering the powers of $\vert \lambda \vert$, this is the worst term appearing in the estimate of
	$\big(\|u_1\|_{H^{7/2}(\Om_1)}+|\la|\|u_1\|_{H^{3/2}(\Om_1)}+\|F\|_\HH\big)\|u_2\|_{H^{3/2}(\Om_2)}$. This is due to the fact that in any other term appearing, the power
	of $\Vert U \Vert_\HH$ is less than $\tfrac{23}{12}$ which results in lower powers of $\vert \lambda \vert$ after applying Young's inequality.
	Now, using $\vert \lambda \vert > \lambda_0$, we can conclude
	\begin{align*}
	\int_I\big|\partial_\nu u_2(q\nabla \overline{u_2})\big|\, {\rm d}S
		& \leq C \big(\|u_1\|_{H^{7/2}(\Om_1)}+|\la|\|u_1\|_{H^{3/2}(\Om_1)}+\|F\|_\HH\big)\|u_2\|_{H^{3/2}(\Om_2)} \\
		& \leq \varepsilon \Vert U \Vert_\HH^2 + C(\varepsilon) \vert \lambda \vert^{60} \Vert F \Vert_\HH^2,
	\end{align*}
	where $\varepsilon > 0$ is arbitrary and $C(\varepsilon) > 0$ is a constant only depending on $\varepsilon.$

\end{proof}

We are now able to finish the proof of Theorem \ref{thm:pol}.

\begin{proof}[Proof of Theorem \ref{thm:pol}]
	By \eqref{pol7}, Lemma \ref{lemma:pol2} and Lemma \ref{lemma:pol3} together with Young's inequality applied to the term
	$\vert \lambda \vert \Vert U \Vert_\HH \Vert F \Vert_\HH,$ we get
	\begin{align*}
		\Vert U \Vert_\HH^2 \leq \varepsilon \Vert U \Vert_\HH^2 + C(\varepsilon) \vert \lambda \vert^{60} \Vert F \Vert_\HH^2
	\end{align*}
	for any $\varepsilon > 0$ and a constant $C(\varepsilon) > 0$ only depending on $\varepsilon.$ This shows
	\[
		\Vert U \Vert_\HH \leq C \vert \lambda \vert^{30} \Vert F \Vert_\HH.
	\]
	Taking $F = 0,$ this estimate also shows that $i\R \cap \sigma_p(\AA) = \emptyset.$ Since $\AA^{-1}$ is compact, the spectrum $\sigma(\AA)$ of $\AA$ coincides with
	the point spectrum $\sigma_p(\AA)$ of $\AA$ and we may
	conclude that $i\R \subset \rho(\AA).$ Now, the assertion follows from Theorem \ref{lemma:pol1}.
\end{proof}



\end{document}